\documentclass[amsmath,secnumarabic,floatfix,amssymb,nofootinbib,nobibnotes,letterpaper,11pt,tightenlines]{revtex4}

\usepackage{times}

\usepackage{geometry}
\usepackage{amssymb}
\usepackage{latexsym, amsmath, amscd,amsthm}
\usepackage{graphicx}
\usepackage[percent]{overpic}
\usepackage{units}
\usepackage{hyperref}
\usepackage{diagrams}
\PassOptionsToPackage{caption=false}{subfig}
\usepackage[lofdepth]{subfig}
\usepackage{booktabs}
\usepackage{mathrsfs}

\usepackage{clrscode}

\newtheorem{theorem}{Theorem}

\newtheorem{lemma}[theorem]{Lemma}
\newtheorem{proposition}[theorem]{Proposition}
\newtheorem{corollary}[theorem]{Corollary}

\theoremstyle{definition}
\newtheorem{definition}[theorem]{Definition}

\newcommand{\R}{\mathbb{R}}
\newcommand{\C}{\mathbb{C}}
\newcommand{\Q}{\mathbb{H}}

\newcommand{\Hopf}{\operatorname{Hopf}}

\newcommand{\Pol}{\operatorname{Pol}}

\newcommand{\Arm}{\operatorname{Arm}}

\newcommand{\cross}{\times}

\newcommand{\chord}{\operatorname{Chord}}

\newcommand{\dVol}{\thinspace\operatorname{dVol}}
\newcommand{\Vol}{\operatorname{Vol}}
\newcommand{\dtheta}{\,\mathrm{d}\theta}
\newcommand{\dphi}{\,\mathrm{d}\phi}
\newcommand{\dr}{\,\mathrm{d}r}
\newcommand{\dx}{\,\mathrm{d}x}
\newcommand{\dy}{\,\mathrm{d}y}
\newcommand{\dz}{\,\mathrm{d}z}
\newcommand{\dt}{\,\mathrm{d}t}
\newcommand{\du}{\,\mathrm{d}u}
\newcommand{\sech}{\operatorname{sech}}

\newcommand{\gyradius}{\operatorname{Gyradius}}

\newcommand{\I}{\mathbf{i}}
\newcommand{\J}{\mathbf{j}}
\newcommand{\K}{\mathbf{k}}
\newcommand{\Beta}{\operatorname{\mathrm{B}}}

\newcommand{\sPol}{{\mathscr{P}\!}}
\newcommand{\sArm}{{\mathscr{A}}}

\newcommand{\vecr}{\vec{r}}

\let\mgp=\marginpar \marginparwidth18mm \marginparsep1mm
\def\marginpar#1{\mgp{\raggedright\tiny #1}}

\let\lbl=\label
\def\label#1{\lbl{#1}\ifinner\else\marginpar{\ref{#1} #1}\ignorespaces\fi}

\begin{document}
\title[]{The Expected Total Curvature of Random Polygons}
\author{Jason Cantarella}
\altaffiliation{University of Georgia, Mathematics Department, Athens GA}
\noaffiliation
\author{Alexander Y. Grosberg}
\altaffiliation{New York University, Physics Department, New York, NY}
\noaffiliation
\author{Robert Kusner}
\altaffiliation{University of Massachussetts, Mathematics Department, Amherst, MA}
\noaffiliation
\author{Clayton Shonkwiler}
\altaffiliation{University of Georgia, Mathematics Department, Athens GA}
\noaffiliation

\begin{abstract} 
We consider the expected value for the total curvature of a random closed polygon. 
Numerical experiments have suggested that as the number of edges becomes large, the difference between the expected total curvature of a random closed polygon and a random open polygon with the same number of turning angles approaches a positive constant. We show that this is true for a natural class of probability measures on polygons, and give a formula for the constant in terms of the moments of the edgelength distribution. 

We then consider the symmetric measure on closed polygons of fixed total length constructed by Cantarella, Deguchi, and Shonkwiler. For this measure, we are able to prove that the expected value of total curvature for a closed $n$-gon is exactly $\frac{\pi}{2} n + \frac{\pi}{4} \frac{2n}{2n-3}$. As a consequence, we show that at least $\nicefrac{1}{3}$ of fixed-length hexagons and $\nicefrac{1}{11}$ of fixed-length heptagons in $\R^3$ are unknotted.    
\end{abstract}
\date{\today}
\maketitle

\section{Introduction} 
\label{sec:introduction}

The study of random polygons is a fascinating topic in geometric probability and statistical physics. Random polygons provide an effective model for long-chain polymers in solution under ``$\theta$-conditions''. There is an essential distinction between open polygonal ``arms'', which are easy to analyze because each edge is sampled independently, and closed polygons, where the closure constraint imposes subtle global correlations between edges. The fixed-length open polygons with $n$ edges in $\R^3$ form a manifold which has a codimension~3 submanifold of closed polygons. To study random polygons, we must first fix a probability measure on open polygons and a corresponding codimension~3 Hausdorff measure on closed polygons. Given these measures, we can then study the statistical properties of the geometry and topology of polygons in each space. 

In recent work, two of us (Cantarella and Shonkwiler)~\cite{cpam} presented a new measure on the space of closed $n$-gons of fixed length constructed using a map from the Stiefel manifold $V_2(\C^n)$ of orthonormal 2-frames in complex $n$-space to the space $\Pol_3(n)$ of closed $n$-gons of length 2. We called this measure the \emph{symmetric measure}. We computed exact expectations of radius of gyration and squared chord lengths with respect to the symmetric measure. These are global geometric invariants of polygons in the sense that they involve edges which are far apart along the polygon. In this paper, we are interested in the total curvature $\kappa$, which is the sum of the turning angles at each vertex of the polygon. This is a local invariant of polygons, in the sense that it is  determined by pairs of adjacent edges. Though we will only consider total curvature in the present work, our methods should also apply to other local geometric invariants such as total torsion, which is determined by triples of adjacent edges.

If we sample the edges in an open polygonal arm independently according to a spherically symmetric distribution, it is easy to see that the expected turning angle at a vertex is $\frac{\pi}{2}$. Thus the expected total curvature of an $n$-edge arm (which has $n-1$ vertices) is $\frac{\pi}{2} (n-1)$. In a closed polygon, the edges are not independently sampled due to the closure constraint. Since the closure constraint involves all the edges, it is reasonable to expect that its effect at any given vertex should become negligible as $n \rightarrow \infty$. Hence we expect that in a closed polygon, the expected turning angle at a vertex should also approach $\frac{\pi}{2}$. This is true, as we will see below.

However, it is not true that $\lim_{n \to \infty} E(\kappa, \Pol_3(n)) - \frac{\pi}{2}n = 0$. In 2007, Plunkett et al.~\cite{Plunkett:2007vx} numerically sampled random closed equilateral polygons of $n$ edges and found that their total curvatures are equal to $\frac{\pi}{2}n + \alpha(n)$, where $\alpha(n)$ tended towards a constant near $1.2$ as $n \to \infty$. In 2008, one of us (Grosberg)~\cite{GrosbergExtra} presented an argument to explain why $\alpha(n) \to \frac{3\pi}{8} = 1.1781 \ldots$ for equilateral polygons. Numerical experiments \cite{cpam} suggested that for the symmetric measure $\alpha(n) \to \frac{\pi}{4}$. This raised an interesting question: why is the asymptotic value of $\alpha(n)$ different for the two measures?

Theorem~\ref{thm:asymptoticCurvature} answers this question. With respect to a probability measure $\mu$ satisfying mild hypotheses, the expected total curvature of a closed $n$-gon is $\frac{\pi}{2} n + \alpha(n)$, where $\alpha(n) \to \frac{3\pi}{8} \frac{m_1^2}{m_2}$. Here, $m_1$ and $m_2$ are first and second moments of edgelength. For random polygons sampled according to the symmetric measure of~\cite{cpam}, this theorem shows that $\alpha(n) \to \frac{\pi}{4}$, in agreement with our previous numerical experiment.

The main result of this paper is that we are able to go much further for the symmetric measure and obtain an exact formula for the expectation of total curvature on $\Pol_3(n)$. To do so, we show in Theorem~\ref{thm:scaleInvariant} that the expectation of any scale-invariant function on $\Pol_3(n)$ in the symmetric measure is equal to the expectation of that function in a new measure called the \emph{Hopf-Gaussian measure}.

The Hopf-Gaussian measure is constructed by applying the Hopf map to a multivariate Gaussian distribution on quaternionic $n$-space. Since the coordinatewise Hopf map is a quadratic form on $\Q^n$ (cf. Section~\ref{sub:quaternionionic_constructions}), the coordinates of edges of polygons sampled according to this measure are differences of chi-squared variables. Hence, they have Bessel distributions. Using this fact, we determine the pdf of the sum of $k$ edges in a random arm in Proposition~\ref{prop:failureToClose}. This enables us to find in Proposition~\ref{prop:pn} an explicit pdf for a pair of edges sampled from a random closed polygon. We then compute the expected value for turning angle by integration in Proposition~\ref{prop:turningangle}. This gives us our main result (Theorem~\ref{thm:exact}): the expectation of total curvature $\kappa$ for a random closed $n$-gon in the symmetric measure $\sigma$ is 
\begin{equation*}
E(\kappa;\Pol_3(n),\sigma) = \frac{\pi}{2} n + \frac{\pi}{4} \frac{2n}{2n-3}.
\end{equation*}

This calculation gives us some new insight into these polygon spaces. For example, consider the old question: what fraction of the space of closed $n$-gons consists of knotted polygons? It has been proved that the fraction of unknots decreases exponentially quickly to zero in various models of random polygons~\cite{JUNGREIS:1994cr,Sumners:1999cd,Diao:2001db}. There are also decades of computational experimentation on this question~(cf. \cite{mrjcp} for references) which show that for small $n$, unknots are very common. Few theorems are known for specific values of $n$. Our total curvature theorem allows us to prove that, as measured by the symmetric measure, at least $\nicefrac{1}{3}$ of the space of fixed-length hexagons in $\R^3$ and $\nicefrac{1}{11}$ of the space of fixed-length heptagons in $\R^3$ consists of unknotted polygons. 

These methods open up a number of new avenues for exploration and experimentation. The ability to write an explicit pdf for pairs, triplets, or other collections of edges raises the hope of computing expectations for other interesting scale-invariant functions, such as total torsion or average crossing number. Numerical integration with respect to these pdfs is also an effective method for approximating expected values. This can give significantly better results than averaging over large ensembles of polygons (cf. Section~\ref{sec:numerical}). This will aid future research on these polygon spaces by allowing conjectured expectations to be tested to high accuracy. 

\section{Asymptotic expected total curvature of polygons}
\label{sec:total-curvature}

The purpose of this section is to compute the asymptotic expected total curvature of closed random polygons under some reasonable hypotheses on the probability measure chosen for polygon space. We will denote the space of $n$-edge open polygons (up to translation) in $\R^d$ by $\sArm_d(n)$ and the subspace of $n$-edge closed polygons in $\R^d$ by $\sPol_d(n)$. Here we do not fix the lengths of the polygons, so $\sArm_d(n)$ consists of vectors of edges 
$(\vec{e_1}, \dots, \vec{e_n}) \in \R^{d} \cross \cdots \cross \R^d = \R^{dn}$ and $\sPol_d(n)$ is the linear codimension $d$ subspace of $\Arm_d(n)$ determined by the closure constraint $\sum \vec{e_i} = \vec{0}$.

We will say that a probability measure $\mu$ on $\sArm_d(n)$ is generated by a spherically symmetric pdf $g$ on $\R^d$ when $\mu$ is the product measure $g \cross \cdots \cross g$. The corresponding probability measure on $\sPol_d(n)$ is the subspace measure with respect to $\mu$. Equivalently, we say that $\mu$ is generated by $g$ if arms are generated by sampling edges independently from $g$ and closed polygons are generated by the same algorithm conditioned on closure.

When the pdf $g$ is spherically symmetric, we can write
\begin{equation}
g(\vecr) \dVol_{\vec{r}} = \frac{1}{\Vol S^{d-1} |\vecr|^{d-1}} f(|\vec{r}|) \dVol_{\vec{r}}
\label{eq:edge_distribution}
\end{equation}
for some non-negative function $f(r)$ so that $\int_0^\infty f(r) \dr = 1$ (and hence $\int_{\R^d} g(\vecr) \dVol_{\vec{r}} = 1$). In this case, the radial moments of $g$ are the ordinary moments of $f$, and both are equal to the moments of edgelength with respect to $\mu$. We denote these by
\begin{equation}
m_p := E(|\vec{e}_i|^p;\mu) = E(|\vec{r}|^p;g) = E(r^p;f).
\label{eq:radial moments}
\end{equation}

A number of standard probability measures on $\sArm_d(n)$ are generated in this way.  For instance, if $f(r) = \delta(r - 1)$, the resulting measure is the standard measure on $n$-edge equilateral arms and the corresponding measure on $\sPol_d(n)$ is the standard measure on closed equilateral polygons. If $f(r) = \mathcal{N}(0,1)$, the resulting measure on $\sArm_d(n)$ is the standard measure on Gaussian random arms and the corresponding measure on $\sPol_d(n)$ is the standard measure on Gaussian random polygons.


If $\mu$ is generated by $g$, it is clear that the expected angle between two edges of a polygon in $\sArm_d(n)$ sampled according to $\mu$ is $\frac{\pi}{2}$, since the edges $\vec{e}_i$ are independently sampled from a spherically symmetric pdf on $\R^d$. Thus, the expected value of total curvature on $\sArm_d(n)$ is given by $\frac{\pi}{2}(n-1)$. Of course, an $n$-edge closed polygon in $\sPol_d(n)$ has an extra turning angle, so we might guess that the expectation of total curvature is $\frac{\pi}{2}n$ instead. In fact, there is a curvature ``surplus'' in a closed polygon. We will now modify the argument in~\cite{GrosbergExtra} to prove

\begin{theorem}\label{thm:asymptoticCurvature}
For $d\geq 2$, if $\mu$ is a measure on $\sArm_d(n)$ generated by a spherically symmetric pdf $g$ which is bounded on $\R^d$ and has finite radial moments $m_1$, $m_2$, and $m_3$ as in~\eqref{eq:radial moments} and we take the corresponding subspace measure on $\sPol_d(n)$, then the expected value of total curvature on $\sPol_d(n)$ approaches 
$\frac{\pi}{2}n + \frac{d}{d-1}\frac{\Beta\left(\nicefrac{d}{2},\nicefrac{d}{2}\right)}{\Beta\left(\nicefrac{(d-1)}{2},\nicefrac{(d+1)}{2}\right)} \frac{m_1^2}{m_2}$ as $n \rightarrow \infty$, where $\Beta$ is the Euler beta function.
\end{theorem}

In particular, when $d=2$ and $3$ we have
\begin{equation*}
	E(\kappa; \sPol_2(n),\mu) \simeq \frac{\pi}{2}n + \frac{4}{\pi} \frac{m_1^2}{m_2} \quad \text{and} \quad E(\kappa; \sPol_3(n),\mu) \simeq \frac{\pi}{2}n + \frac{3\pi}{8} \frac{m_1^2}{m_2}
\end{equation*}
for large $n$.

\begin{proof}[Proof of Theorem~\ref{thm:asymptoticCurvature}]
We will assume for the duration of the proof that $\sArm_d(n)$ (for any $n$) has a fixed probability measure $\mu$ generated by a fixed spherically symmetric pdf $g$ on $\R^d$ given by $g(\vecr) = f(|\vecr|)/\Vol S^{d-1} |\vecr|^{d-1}$ as in~\eqref{eq:edge_distribution}, and that $\sPol_d(n)$ has the subspace measure induced by $\mu$.
  
Let the Green's function $G_k(\vec{r})$ be the probability density of the end-to-end vector $\vec{r}$ in $\sArm_d(k)$ with respect to $\dVol_{\vec{r}}$. We can write this explicitly as 
\begin{equation*}
G_k(\vec{r}) = \int g(\vec{e}_1) \cdots g(\vec{e}_k) \, \delta(\vec{e}_1 + \cdots + \vec{e}_k - \vec{r}) \dVol_{\vec{e}_1} \cdots \dVol_{\vec{e}_k}.
\end{equation*}

If we consider the joint probability distribution of all edges in a closed polygon, we can treat it as a conditional probability on a set of edge vectors $\vec{e}_1, \ldots , \vec{e}_n$ conditioned on the closure constraint $\sum \vec{e}_i = 0$. This conditional probability can then be written as
\[
	P(\vec{e}_1, \ldots , \vec{e}_n) \dVol_{\vec{e}_1} \cdots \dVol_{\vec{e}_n} = \frac{g(\vec{e}_1) \cdots g(\vec{e}_n) \, \delta(\vec{e}_1 + \ldots + \vec{e}_n)}{C_n} \dVol_{\vec{e}_1} \cdots \dVol_{\vec{e}_n},
\]
where $C_n= G_n(\vec{0})$ is the codimension $d$ Hausdorff measure of the closed polygon space $\sPol_d(n)$.

Given this joint distribution on all the edges, we can integrate out all but two of the edges to get the joint probability distribution on two consecutive edges, $P(\vec{e}_i, \vec{e}_{i+1})$. Since this is independent of $i$, we may as well consider the case $i=1$:
\begin{equation}\label{eq:abstractPairwisePDF}
	P(\vec{e}_1, \vec{e}_2) \dVol_{\vec{e}_1} \dVol_{\vec{e}_2} = g(\vec{e}_1) g(\vec{e}_2) \frac{G_{n-2}(-\vec{e}_1 - \vec{e}_2)}{C_n}  \dVol_{\vec{e}_1} \dVol_{\vec{e}_2}.
\end{equation}
In other words, in order for the edges $\vec{e}_1$ and $\vec{e}_2$ to come from a closed polygon, the remaining $n-2$ edges must connect the head of $\vec{e}_2$ to the tail of $\vec{e}_1$. 

Finding exact expressions for $G_{n-2}$ and $C_n$ is quite challenging in general, but we can approximate both fairly easily by observing that the failure-to-close vector for an element of $\sArm_d(k)$ is just the sum of the edges. Given that we are sampling edges of our arms independently, that the third moment of $g$ is finite, and that $g$ is a bounded density on $\R^d$, the vector local limit theorem of Bikjalis~\cite{bikjalis} implies that the pdf of the (normalized) failure-to-close distribution converges in sup norm to the pdf of a normal distribution.  

We can recover the parameters of this normal distribution by noting that spherical symmetry implies that the mean of the failure-to-close distribution is zero, the variance of each coordinate of an edge vector is $\nicefrac{m_2}{d}$, and the coordinates of an edge vector are uncorrelated. Therefore, the pdf of $\frac{1}{\sqrt{k}} G_k(\vec{r})$ converges in sup norm to the pdf of the $d$-dimensional normal distribution 
\[
	\mathcal{N}\left(\vec{0},\operatorname{diag}\left(\nicefrac{m_2}{d}, \ldots , \nicefrac{m_2}{d}\right)\right),
\]
where $\operatorname{diag}(a_1,\ldots, a_k)$ is the diagonal $k \times k$ matrix with entries $a_1,\ldots , a_k$. 
%
%

In particular, as $n \rightarrow \infty$, we have that $G_{n-2}$ and $C_n = G_n(\vec{0})$ are asymptotic in sup norm to
\begin{align*}
	G_{n-2}(\vec{r}) & \simeq \left( \frac{d}{2\pi (n-2) m_2} \right)^{\nicefrac{d}{2}} \hspace{-0.1in} \exp\left({\frac{-dr^2}{2(n-2) m_2}}\right) \\
	C_n & \simeq \left( \frac{d}{2\pi \,n \,m_2} \right)^{\nicefrac{d}{2}} \hspace{-0.15in},
\end{align*}
where $r = |\vec{r}|$. From~\eqref{eq:abstractPairwisePDF}, then, we see that the pdf $P(\vec{e_1},\vec{e_2})$ is sup norm close to the function
\begin{equation}
	P(\vec{e}_1,\vec{e}_2) \simeq g(\vec{e}_1)g(\vec{e}_2) \left(\frac{n}{n-2}\right)^{\nicefrac{d}{2}} \hspace{-0.1in} \exp\left({\frac{-d|\vec{e}_1+\vec{e}_2|^2}{2(n-2)m_2}}\right).
\label{eq:p approximation}	
\end{equation}

Let $\theta(\vec{e}_1, \vec{e}_2)$ be the angle between $\vec{e}_1$ and $\vec{e}_2$, which is to say the turning angle between the two edges. We will now prove that $E(\theta) \rightarrow \frac{\pi}{2} +\frac{d}{n(d-1)}\frac{\Beta\left(\nicefrac{d}{2},\nicefrac{d}{2}\right)}{\Beta\left(\nicefrac{(d-1)}{2},\nicefrac{(d+1)}{2}\right)} \frac{m_1^2}{m_2}$ as $n \rightarrow \infty$. 

First, for any $\epsilon > 0$ we may choose $\rho$ so that the integral of $\theta(\vec{e}_1,\vec{e}_2) P(\vec{e}_1,\vec{e}_2)$ over the complement of the ball $B(\rho)$ of radius $\rho$ centered at the origin obeys
\begin{equation}
\int\limits_{\R^{2d} - B(\rho)} \hspace{-0.15in} \theta(\vec{e}_1,\vec{e}_2) P(\vec{e}_1,\vec{e}_2) \dVol_{\vec{e}_1} \dVol_{\vec{e}_2} < \epsilon.
\label{eq:bigball}
\end{equation}
To see this, observe that $P(\vec{e}_1,\vec{e}_2)$ is a pdf on $\R^{2d}$, so its improper integral over the entire space converges. This means that the $L^1$ norm of $P$ on the complement of a ball of radius $\rho$ goes to $0$ as $\rho \rightarrow \infty$. But $\theta$ is bounded by $\pi$ so $\|\theta(\vec{e}_1,\vec{e}_2) P(\vec{e}_1,\vec{e}_2)\|_1 \leq \pi \| P(\vec{e}_1,\vec{e}_2) \|_1$ where the norms are over the complement of the ball $B(\rho)$. Choosing $\rho$ large enough that the rhs is less than $\epsilon$ yields~\eqref{eq:bigball}.

Similarly, $g(\vec{e_1}) g(\vec{e_2})$ is the pdf of a two-edge arm, so its improper integral over $\R^{2d}$ converges as well. Since $g$ has finite first and second moments on $\R^d$, this product has finite mixed moments of order up to 2 on $\R^{2d}$. In particular, for any quadratic polynomial $q(r_1,r_2)$ with coefficients bounded by $\pm \lambda^2$ we may choose $\rho(\lambda^2)$ so that we have 
\begin{equation}
\int\limits_{\R^{2d} - B(\rho)} \hspace{-0.15in} q(r_1,r_2) \, g(\vec{e_1}) g(\vec{e_2}) \dVol_{\vec{e}_1} \dVol_{\vec{e}_2} < \epsilon. 
\label{eq:bigball2}
\end{equation}

We now turn to the interior of the ball. Since the pair $(\vec{e}_1,\vec{e}_2)$ is in the interior of $B(\rho)$ in~$\R^{2d}$, we may choose $n$ large enough that $\frac{|\vec{e}_1  + \vec{e}_2|^2}{(n-2)m_2}$ is as close to zero as we like. In particular, we may choose $n$ large enough that the exponential in~\eqref{eq:p approximation} is sup norm close to its linear Taylor approximation on the entire ball. Further, we can approximate $\left(\frac{n}{n-2}\right)^{\nicefrac{d}{2}}$ by $1 + \frac{d}{n}$ and $\frac{1}{n-2}$ by $\frac{1}{n}$ and we have $P(\vec{e}_1,\vec{e}_2)$ sup norm close to the following function over the entire ball:
\begin{equation}
	P(\vec{e}_1,\vec{e}_2) \simeq g(\vec{e}_1)g(\vec{e}_2) \left[1 + \frac{d}{n} - \frac{d|\vec{e}_1|^2}{2n m_2} - \frac{d|\vec{e}_2|^2}{2n m_2} - \frac{d\langle \vec{e}_1, \vec{e}_2\rangle}{n m_2}\right] .
\label{eq:p taylor approximation}
\end{equation}

The expected value of $\theta(\vec{e}_1,\vec{e}_2)$ is just
\begin{multline}
	\int\limits_{\R^{2d}} \theta(\vec{e}_1,\vec{e}_2) P(\vec{e}_1, \vec{e}_2) \dVol_{\vec{e}_1}\! \dVol_{\vec{e}_2} \\
	= \int\limits_{B(\rho)}  \theta(\vec{e}_1,\vec{e}_2) P(\vec{e}_1, \vec{e}_2) \dVol_{\vec{e}_1}\! \dVol_{\vec{e}_2} + \hspace{-0.1in}
	\int\limits_{\R^{2d} - B(\rho)} \hspace{-0.1in} \theta(\vec{e}_1,\vec{e}_2) P(\vec{e}_1, \vec{e}_2) \dVol_{\vec{e}_1}\! \dVol_{\vec{e}_2}\!.
	\label{eq:E theta}
\end{multline}
By~\eqref{eq:bigball}, the second integral on the right is small and we can ignore it. Since the ball is a bounded domain, the fact that $P(\vec{e}_1,\vec{e}_2)$ is sup norm close to the approximation in~\eqref{eq:p taylor approximation} tells us that the integral of the bounded function $\theta$ against the approximation is close to the first integral on the right. 

Consider the approximation~\eqref{eq:p taylor approximation} to $P(\vec{e}_1,\vec{e}_2)$. For large enough $n$, the quantity in square brackets is a quadratic polynomial in $r_1$ and $r_2$ with coefficients between $-2$ and $2$, so the inequality~\eqref{eq:bigball2} applies with $\rho \geq \rho(\sqrt{2})$. Hence, since $\theta$ is bounded, its integral against the approximation over the ball is close to its integral against the approximation over all of $\R^{2d}$. In other words, we can approximate the first integral on the rhs of~\eqref{eq:E theta} by the integral
\[
	\int_{\R^{2d}} \theta(\vec{e}_1,\vec{e}_2) g(\vec{e}_1)g(\vec{e}_2) \left[1 + \frac{d}{n} - \frac{d|\vec{e}_1|^2}{2n m_2} - \frac{d|\vec{e}_2|^2}{2n m_2} - \frac{d\langle \vec{e}_1, \vec{e}_2\rangle}{n m_2}\right] \dVol_{\vec{e}_1}\dVol_{\vec{e}_2}.
\]

We now evaluate the above integral. We will write $\vec{e}_1$ and $\vec{e}_2$ in spherical coordinates. Since the integrand is spherically symmetric, we can integrate out the angular coordinates of, say, $\vec{e}_1$ and assume that $\vec{e}_1$ lies along the $z$-axis. This produces a factor of $\Vol S^{d-1}$. Since the rotation of $\vec{e}_2$ in the $(d-1)$-plane perpendicular to the $\vec{e}_1$-axis does not change $\theta$ or the approximation to $P(\vec{e}_1,\vec{e}_2)$, we can integrate out another $\Vol S^{d-2}$. Since $\theta$ is now the polar angle for $\vec{e}_2$ and $\frac{\Vol S^{d-2}}{\Vol S^{d-1}} = \frac{\Gamma\left(\nicefrac{d}{2}\right)}{\sqrt{\pi}\Gamma\left(\nicefrac{(d-1)}{2}\right)}$, the integral reduces to
\begin{multline*}
	\frac{\Gamma\left(\nicefrac{d}{2}\right)}{\sqrt{\pi}\Gamma\left(\nicefrac{(d-1)}{2}\right)} \int\limits_0^\pi \int\limits_0^\infty \int\limits_0^\infty \theta\, \frac{f(r_1)}{r_1^{d-1}} \frac{f(r_2)}{r_2^{d-1}} \left[ 1 + \frac{d}{n} - \frac{dr_1^2}{2nm_2} - \frac{dr_2^2}{2nm_2} \right.\\
	\left. - \frac{dr_1r_2}{nm_2} \cos \theta \right] r_1^{d-1} r_2^{d-1} \sin^{d-2} \theta \, \dr_1 \dr_2 \dtheta.
\end{multline*}
Since $\int_0^\infty f(r_i) \dr_i = 1$ and $\int_0^\infty r_i^p f(r_i) \dr_i = m_p$, integrating out $r_1$ and $r_2$ yields
\begin{equation*}
	\frac{\Gamma\left(\nicefrac{d}{2}\right)}{\sqrt{\pi}\Gamma\left(\nicefrac{(d-1)}{2}\right)} \int_0^\pi \theta  \left(1 - \frac{dm_1^2}{nm_2}\cos \theta \right)\sin^{d-2} \theta \dtheta  = \frac{\pi}{2} + \frac{d}{n(d-1)}\frac{\Beta\left(\nicefrac{d}{2},\nicefrac{d}{2}\right)}{\Beta\left(\nicefrac{(d-1)}{2},\nicefrac{(d+1)}{2}\right)} \frac{m_1^2}{m_2}
\end{equation*}
after integrating by parts.

Since this is the expected value of the turning angle between two edges of the polygon, multiplying by $n$ yields the desired expression for expected total curvature.
\end{proof}

\section{The Symmetric and Hopf-Gaussian measures on polygon spaces} 
\label{sec:new-measure}

Theorem~\ref{thm:asymptoticCurvature} applies to a broad class of measures on polygon space, but not to certain highly symmetric measures defined in \cite{cpam}. These symmetric measures on fixed-length polygons in space and in the plane are interesting for a number of reasons: they come from a natural geometric construction, expectations and moments of chordlengths and radii of gyration are exactly computable and scale like the corresponding expectations for equilateral polygons, and there is an algorithm for direct sampling from these measures which is fast (linear in the number of edges) and easy to code. Our goal for the rest of the paper is to determine the expected total curvature of polygons with respect to these measures.

The symmetric measure is most naturally defined on the space of (open or closed) $n$-gons of fixed total length~$2$ in either $\R^3$ or $\R^2$, which we denote by $\Arm_d(n)$ for open polygons and $\Pol_d(n)$ for closed polygons in $\R^d$. Of course we can extend the definition to the space of polygons of any fixed length by scaling. Viewed as a subspace of $\sArm_d(n)$, the space $\Pol_d(n)$ differs from $\sPol_d(n)$ in that elements satisfy constraints on both closure \emph{and} total length. Therefore, Theorem~\ref{thm:asymptoticCurvature} does not apply to the symmetric measure.

However, since total curvature is a scale-invariant quantity and since $\sPol_d(n)$ is a cone over $\Pol_d(n)$, the expected total curvature of polygons in $\sPol_d(n)$ -- which we \emph{can} determine asymptotically using Theorem~\ref{thm:asymptoticCurvature} -- will be the same as the expected total curvature of polygons in $\Pol_d(n)$ provided that this expectation is computed with respect to a measure on $\sPol_d(n)$ which is a product of some measure on the cone parameter and the symmetric measure on $\Pol_d(n)$. Indeed, in Section~\ref{sub:gaussians} we will define the \emph{Hopf-Gaussian measure} $\operatorname{H}$ on $\sArm_d(n)$ and $\sPol_d(n)$ for $d = 2, 3$ by applying the Hopf map to the standard multivariate Gaussian measure on $\Q^n$. The Hopf-Gaussian measure on $\sArm_d(n)$ turns out to be the product $\operatorname{H} = \chi_{2^{d-1}n}^2 \times \sigma$, where $\chi_{2^{d-1}n}^2$ is the chi-squared distribution with $2^{d-1}n$ degrees of freedom on the interval $[0,+\infty)$ which parametrizes the cone direction and $\sigma$ is the symmetric measure on $\Arm_d(n)$. Likewise, the Hopf-Gaussian measure on $\sPol_d(n)$ is the product $\operatorname{H} = \chi_{2^{d-1}n}^2 \times \sigma$, where now $\sigma$ is the symmetric measure on $\Pol_d(n)$. 

An immediate consequence of this construction is the following theorem, which is the central message of this section:

\begin{theorem}
Suppose $F: \sArm_d(n) \to \R$ is a scale-invariant function. Then the expected value of $F$ over $\sArm_d(n)$ with respect to the Hopf-Gaussian measure $\operatorname{H}$ is the same as the expected value of $F$ over $\Arm_d(n)$ with respect to the symmetric measure $\sigma$; that is
\[
	E(F;\sArm_d(n),\operatorname{H}) = E(F;\Arm_d(n),\sigma).
\]
Likewise, if $F: \sPol_d(n) \to \R$ is scale-invariant, then
\[
	E(F; \sPol_d(n), \operatorname{H}) = E(F; \Pol_d(n), \sigma).
\]
\label{thm:scaleInvariant}
\end{theorem}

As we will see, the Hopf-Gaussian measure satisfies the hypotheses of Theorem~\ref{thm:asymptoticCurvature}, so the combination of Theorems~\ref{thm:asymptoticCurvature} and~\ref{thm:scaleInvariant} will allow us to determine the expected asymptotic total curvature on $\Pol_d(n)$ with respect to $\sigma$ from the first and second moments of edgelength on $\sArm_d(n)$ with respect to $\operatorname{H}$, which we compute in Section~\ref{sub:moments}. These asymptotic total curvature expectations are given by:

\begin{corollary}
	For $d \in \{2,3\}$ and large $n$, the expected total curvature on $\Pol_d(n)$ with respect to the symmetric measure is
	\[
		E(\kappa; \Pol_2(n),\sigma) \simeq \frac{\pi}{2}n + \frac{2}{\pi}, \quad E(\kappa; \Pol_3(n),\sigma) \simeq \frac{\pi}{2}n + \frac{\pi}{4}.
	\]
	\label{cor:asymptoticTotalCurvature}
\end{corollary}

The value of $\frac{\pi}{4}$ for the total curvature surplus of polygons in $\R^3$ agrees with our numerical experiments in \cite{cpam}.

Of course, Theorem~\ref{thm:scaleInvariant} applies to any scale-invariant functional on polygons, not just total curvature. We expect that it will be useful for determining the expected values of other interesting quantities such as total torsion and average crossing number.

\subsection{Quaternions and the Symmetric Measure on Polygon Spaces} 
\label{sub:quaternionionic_constructions}

In this subsection we recall the construction of the symmetric measure on $\Arm_d(n)$ and $\Pol_d(n)$ from \cite{cpam}. Recall that these are spaces of arms and polygons of fixed total length 2. In principle everything could be scaled to any desired fixed length, but the choice of length 2 will be the most convenient. Since the translation of the following definitions and results to any other scale is straightforward, we will not discuss this scaling further.

\begin{definition}\label{def:Arm}
	For $d\in \{2,3\}$, let $\Arm_d(n)$ be the moduli space of $n$-edge polygonal arms (which may not be closed) of length 2 up to translation in $\R^d$. An element of $\Arm_d(n)$ is a list of edge vectors $\vec{e}_1, \ldots , \vec{e}_n \in \R^d$ whose lengths sum to $2$. 
\end{definition} 

Consider the Hopf map from the division algebra of quaternions $\Q$ to the space of imaginary quaternions (which we identify with $\R^3$) given by
\[
	\Hopf(q) = \bar{q}\I q,
\]
where $\bar{q}$ is the quaternionic conjugate of $q$. In coordinates, if $q = (q_0,q_1,q_2,q_3)$, then
\begin{equation}\label{eq:hopfDef}
	\Hopf(q) = (q_0^2 + q_1^2 - q_2^2 - q_3^2, 2q_1q_2 - 2q_0q_3, 2q_0q_2+2q_1q_3).
\end{equation} 
We extend the Hopf map coordinatewise to a map $\Hopf: \Q^n \to \R^{3n}$. Then $\Hopf$ is a smooth map from the sphere $S^{4n-1}$ of radius $\sqrt{2}$ in $\Q^n$ onto $\Arm_3(n)$. Specifically, for $\vec{q} \in S^{4n-1}$ the edge set of the polygon $\Hopf(\vec{q})$ is
\[
	(\vec{e}_1, \ldots , \vec{e}_n) := (\Hopf(q_1), \ldots , \Hopf(q_n)).
\]
We call $S^{4n-1}$ the \emph{model space} for $\Arm_3(n)$.

Similarly, the restriction of $\Hopf$ to the $1\oplus \J$ and $\I \oplus \K$ planes gives a map to the $\I \oplus \K$ plane, which we identify with $\C = \R^2$. Specifically,
\begin{align*}
	\Hopf(a+b\J) & = \I \overline{(a+b\J)^2} \\
	\Hopf(a\I + b\K) & = \I(a+b\J)^2.
\end{align*}
In other words, if $z = a+b\J$, then $\Hopf(z) = \I \bar{z}^2$ and $\Hopf(\I z) = \I z^2$. Extending this map coordinatewise yields a smooth, surjective map $\Hopf: S^{2n-1} \sqcup S^{2n-1} \to \Arm_2(n)$; consequently, the disjoint union $S^{2n-1} \sqcup S^{2n-1}$ is the model space for $\Arm_2(n)$.

\begin{definition}\label{def:Pol}
	For $d \in \{2,3\}$, let $\Pol_d(n)$ be the moduli space of closed $n$-gons of length 2 up to translation in $\R^d$.
\end{definition}

 Since $\Pol_d(n) \subset \Arm_d(n)$, the inverse image $\Hopf^{-1}(\Pol_d(n))$ is well-defined and will be the model space for $\Pol_d(n)$. To describe this model space for $d=3$ as a subset of  $S^{4n-1}$, the sphere of radius $\sqrt{2}$ in $\Q^n$, it is convenient to write $S^{4n-1}$ as the join $S^{2n-1} \star S^{2n-1}$. In coordinates, the join map is given by
\[
	(\vec{u},\vec{v},\theta) \mapsto \sqrt{2}(\cos \theta \vec{u}+ \sin \theta \vec{v} \J)
\]
where $\vec{u},\vec{v} \in \C^n$ lie on the unit sphere and $\theta \in [0,\nicefrac{\pi}{2}]$. The Stiefel manifold $V_2 (\C^n)$ of Hermitian orthonormal 2-frames $(\vec{u},\vec{v})$ in $\C^n$ can be identified with the subspace
\[
	\{(\vec{u},\vec{v},\nicefrac{\pi}{4}) : \langle \vec{u}, \vec{v} \rangle = 0 \} \subset S^{4n-1}
\]
and, as Hausmann and Knutson first observed \cite{Knutson:2_iyExxE}, this manifold is precisely the model space for $\Pol_3(n)$.

\begin{proposition}[\cite{Knutson:2_iyExxE}]\label{prop:HopfPol}
	The coordinatewise Hopf map takes $V_2(\C^n) \subset \C^n \times \C^n = \Q^n$ onto $\Pol_3(n)$. 
\end{proposition}

The key to proving the above proposition is to note that the Hopf map applied to a quaternion $q$ can be written more simply by letting $q = a + b \J$ for $a,b \in \C$:
\[
	\Hopf(q) = \Hopf(a+b\J) = (|a|^2 - |b|^2, 2\Im(a\bar{b}), 2\Re(a\bar{b})) = \I(|a|^2 - |b|^2 + 2a \bar{b}\J).
\]

Let $V_2(\R^n)$ be the real Stiefel manifold of orthonormal 2-frames in $\R^n$, which sits naturally in $\R^n \oplus \J\R^n$. A result analogous to the above holds for planar polygons:

\begin{proposition}[\cite{Knutson:2_iyExxE}]\label{prop:HopfPlanarPol}
	The coordinatewise Hopf map takes $V_2(\R^n) \sqcup \I V_2(\R^n)$ onto $\Pol_2(n)$.
\end{proposition}

With these maps in place, we can define probability measures on the arm and polygon spaces by pushing forward measures on the model spaces. Since the model spaces are homogeneous spaces, it is natural to push forward Haar measure on the model spaces, which is what we did in \cite{cpam}. Since Haar measure is also the measure defined by the standard Riemannian metrics on these spaces, this gives the following definition of the \emph{symmetric measure} $\sigma$ on $\Pol_3(n)$:
\[
	\sigma(U) = \frac{1}{\Vol V_2(\C^n)} \int_{\Hopf^{-1}(U)} d\Vol_{V_2(\C^n)} \quad \text{for } U \subset \Pol_3(n).
\]
The symmetric measures on the other arm and polygon spaces are defined analogously.

The space $\Pol_d(n)$ is topologically the union of the spaces of polygons with fixed edgelengths $r_1, \ldots , r_n$ such that $\sum r_i = 2$, so any expectation over $\Pol_d(n)$ with respect to the symmetric measure is a weighted average of the expectations over these spaces. In future work we intend to determine how the average is weighted and with respect to which measure on the fixed edgelength spaces. For equilateral polygons the answer is simple and pleasant: the restriction of the symmetric measure to the subspace of equilateral polygons is just the natural measure on this space, namely the subspace measure on $n$-tuples of vectors in the round $S^2$ which sum to zero.


\subsection{The Hopf-Gaussian Measure on Polygon Spaces} 
\label{sub:gaussians}

We would like to compute the expectation of total curvature over $\Pol_d(n)$ with respect to the symmetric measure. Unfortunately, since elements of $\Pol_d(n)$ are chosen from $\sArm_d(n)$ by conditioning on the polygon both being closed and having total length 2, we cannot directly apply Theorem~\ref{thm:asymptoticCurvature}. However, since the total curvature $\kappa$ is scale-invariant and since $\sPol_d(n)$ is a cone over $\Pol_d(n)$ (with the cone direction parametrized by the length of the polygon), we have that
\[
	E(\kappa; \sPol_d(n),\mu) = E(\kappa; \Pol_d(n),\sigma)
\]
for any measure $\mu$ on $\sPol_d(n)$ such that $\mu = \rho \times \sigma$ for some measure $\rho$ on $[0,+\infty)$.

At the level of model spaces, the picture is clearer. $\Hopf$ maps $\Q^n$ onto $\sArm_3(n)$ and the image of the sphere of radius $r$ is exactly the copy of $\Arm_d(n)$ consisting of polygonal arms with total length $r^2$. Moreover, the measure on this scaled copy of $\Arm_d(n)$ is exactly the pushforward of the standard measure on the sphere of radius $r$, so it is the symmetric measure defined in the previous section. Therefore, we can define a measure on $\sArm_3(n)$ which is the product of some measure on $[0,+\infty)$ and the symmetric measure on $\Arm_3(n)$ simply by pushing forward \emph{any} spherically symmetric measure on $\Q^n$. Of course, if we may choose any spherically symmetric measure, the obvious choice is the multivariate Gaussian measure: this is both a spherically symmetric measure and a product measure on the coordinates and we can expect that the fact that the individual coordinate distributions are Gaussian will simplify our computations considerably.

In the case of planar polygons, the model space for $\sArm_2(n)$ is the explicit copy of $\C^n \cup \C^n$ given by $(\R^n\oplus\J\R^n) \cup (\I\R^n\oplus\K\R^n)$, so we will push forward the Gaussian measure on $\C^n$:

\begin{definition}\label{def:measures}
	If $\gamma^{4n}$ is the standard Gaussian measure on $\Q^n = \R^{4n}$, then the \emph{Hopf-Gaussian measure} $\operatorname{H}$ on $\sArm_3(n)$ is defined by
	\[
		\operatorname{H} (U)  = \int_{\Hopf^{-1}(U)} d\gamma^{4n} \quad \text{for } U \subset \sArm_3(n).
	\]
	Likewise, if $\gamma^{2n}$ is the measure on $(\R^n\oplus\J\R^n) \cup (\I\R^n\oplus\K\R^n)$ naturally induced by the standard Gaussian measure on $\C^n = \R^{2n}$, then the \emph{Hopf-Gaussian measure} $\operatorname{H}$ on $\sArm_2(n)$ is defined by
	\[
		\operatorname{H}(U)  = \int_{\Hopf^{-1}(U)} d\gamma^{2n} \quad \text{for } U \subset \sArm_2(n).
	\]
\end{definition}

The fact that the multivariate Gaussian is a product measure implies that:

\newpage

\begin{proposition}\label{prop:HopfGaussianGenerated}
	The Hopf-Gaussian measure on $\sArm_3(n)$ is generated by the pdf
	\[
		g(\vecr) = \frac{e^{-\nicefrac{|\vecr|}{2}}}{16 \pi |\vecr|}
	\]
	and the Hopf-Gaussian measure on $\sArm_2(n)$ is generated by the pdf
	\[
		g(\vecr) = \frac{e^{-\nicefrac{|\vecr|}{2}}}{4\pi |\vecr|}.
	\]
\end{proposition}

\begin{proof}
	Since $\gamma^{4n}$ is a product measure on $\Q^n$, its restriction to each $\Q^1$ factor is the standard four-dimensional Gaussian. In particular, for $\vec{q} = (q_1, \ldots , q_n) \in \Q^n$ sampled according to $\gamma^{4n}$, the $q_i \in \Q$ are independent, identically distributed, and spherically symmetric. Therefore, the edges of the polygon $\Hopf(\vec{q}) = (\Hopf(q_1), \ldots , \Hopf(q_n))$ are independent, identically distributed, and, since the Hopf map is $SU(2)$-equivariant, spherically symmetric. Therefore, the Hopf-Gaussian measure $\operatorname{H}$ on $\sArm_3(n)$ is generated by a spherically symmetric distribution on $\R^3$, which we now determine. 
	
	The four real components of each $q_i$ are themselves Gaussian-distributed. Therefore, since $|\Hopf(q_i)| = |q_i|^2$, each edgelength of a Hopf-Gaussian polygon is given by the sum of the squares of four Gaussian real numbers, so these edgelengths follow a chi-squared distribution with four degrees of freedom. Thus, the pdf of the edgelength distribution is
	\[
		f(r) = \frac{ r e^{-\nicefrac{r}{2}}}{4}.
	\]
	But then the pdf $g$ of the spherically symmetric edge distribution is 
	\[
		g(\vecr) = \frac{1}{4\pi |\vecr|^2} f(|\vecr|) = \frac{e^{-\nicefrac{|\vecr|}{2}}}{16 \pi |\vecr|},
	\]
	as desired.
	
	The proof in the 2-dimensional case is completely parallel.
\end{proof}

To identify the model space for $\sPol_d(n)$, note that, as with the $\Pol_d(n)$ and $\Arm_d(n)$ spaces, we have that $\sPol_d(n) \subset \sArm_d(n)$. Therefore, $\Hopf^{-1}(\sPol_d(n))$ is a well-defined subset of $\Q^n$ which will be the model space for $\sPol_d(n)$. We can identify this model space more explicitly as follows. Focusing on the case $d=3$ for the moment, since $\Q^n$ is the cone over $S^{4n-1}$ it is convenient to write $\Q^n$ as the cone of the join $S^{2n-1} \star S^{2n-1}$. In coordinates, the ``cone-join'' map is given by
\[
	(\vec{u}, \vec{v}, \theta, s) \mapsto s (\cos \theta \vec{u} + \sin \theta \vec{v} \J)
\]
where $\vec{u},\vec{v} \in \C^n$ are unit vectors, $\theta \in [0,\nicefrac{\pi}{2}]$ and $s \in [0,+\infty)$. The cone $C V_2 (\C^n)$ over the Stiefel manifold $V_2 (\C^n)$ can then be identified with the subspace
\[
	\{(\vec{u},\vec{v},\nicefrac{\pi}{4},s): \langle \vec{u},\vec{v}\rangle = 0\} \subset \Q^n
\]
and the proof of Proposition~\ref{prop:HopfPol} generalizes to show that $CV_2(\C^n)$ is the model space for $\sPol_3(n)$:

\begin{proposition}\label{prop:HopfsPol}
	The coordinatewise Hopf map takes $CV_2(\C^n) \subset \Q^n$ onto $\sPol_3(n)$.
\end{proposition}

If $CV_2(\R^n)$ is the cone over the real Stiefel manifold $V_2(\R^n)$, then the same reasoning yields the analogue of Proposition~\ref{prop:HopfPlanarPol}:

\begin{proposition}\label{prop:HopfPlanarsPol}
	The coordinatewise Hopf map takes $CV_2(\R^n) \cup \I CV_2(\R^n) \subset \Q^n$ onto $\sPol_2(n)$. 
\end{proposition}


Since $\sPol_d(n) \subset \sArm_d(n)$, we can define the \emph{Hopf-Gaussian measure} on $\sPol_d(n)$ as the subspace measure inherited from the Hopf-Gaussian measure on $\sArm_d(n)$ from Definition~\ref{def:measures}. 

To prove Theorem~\ref{thm:scaleInvariant}, which says that the expected value of any scale-invariant function on polygons is the same whether we compute it with respect to the Hopf-Gaussian measure or the symmetric measure, it suffices to show that the Hopf-Gaussian measure is a product measure:

\begin{proposition}\label{prop:productMeasures}
	Suppose $d = 2$ or $3$. Then the Hopf-Gaussian measure $\operatorname{H}$ on $\sArm_d(n)$ is the product $\chi_{2^{d-1}n}^2 \times \sigma $ of the chi-squared distribution with $2^{d-1}n$ degrees of freedom on the interval $[0,+\infty)$ and the symmetric measure on $\Arm_d(n)$.
	
	Likewise, the Hopf-Gaussian measure on $\sPol_d(n)$ is the product $\chi_{2^{d-1}n}^2 \times \sigma$ of the chi-squared distribution with $2^{d-1}n$ degrees of freedom on $[0,+\infty)$ and the symmetric measure on $\Pol_d(n)$.
\end{proposition}

\begin{proof}
	Since the Gaussian measure on $\Q^n = \R^{4n}$ is $SO(4n)$-equivariant, its restriction to the sphere $S^{4n-1}(r)$  of radius $r$ is, after normalization, just the uniform probability measure on the sphere. Since $\Hopf(S^{4n-1}(r))$ is the space of arms of total length $r^2$, this means that the restriction of the Hopf-Gaussian measure on $\sArm_3(n)$ to this space is, after normalization, the symmetric probability measure $\sigma$ defined in Section~\ref{sub:quaternionionic_constructions}. Likewise, the restriction of the Hopf-Gaussian measure on $\sArm_2(n)$ to planar arms of total length $r^2$ is just the symmetric measure.
	
	Therefore, the measure on $\sArm_d(n)$ is the product $\rho \times \sigma$ for some measure $\rho$ on the interval $[0,+\infty)$, so it suffices to see that $\rho$ is the chi-squared distribution. Since the interval parametrizes the total length of a polygon, we need to analyze the distribution of total length of polygonal arms. For $\vec{q} = (q_1, \ldots , q_n) \in \Q^n$, the arm $\Hopf(\vec{q}) \in \sArm_3(n)$ has total length
	\[
		\sum \left|\Hopf(q_i)\right| = \sum |\bar{q}_i \I q_i| =  \sum |q_i|^2.
	\]
	Since each $|q_i|^2$ is the sum of the squares of four standard Gaussians, the total length of the polygon $\Hopf(\vec{q})$ follows the standard chi-squared distribution with $4n$ degrees of freedom. Therefore, the measure $\rho$ on $[0,+\infty)$ is the measure induced by the chi-squared distribution with $4n$ degrees of freedom. Likewise, for polygonal arms in the plane, the measure on total length is induced by the standard chi-squared distribution with $2n$ degrees of freedom, since in that case each $|q_i|^2$ is the sum of the squares of two standard Gaussians.
	
	The fact that the Hopf-Gaussian measure on $\sPol_d(n)$ is the product of the chi-squared distribution on $[0,+\infty)$ and the symmetric measure on $\Pol_d(n)$ then follows immediately from the definition of the Hopf-Gaussian measure on $\sPol_d(n) \subset \sArm_d(n)$ as the subspace measure and the fact that the symmetric measure on $\Pol_d(n) \subset \Arm_d(n)$ is the subspace measure.
\end{proof}

Theorem~\ref{thm:scaleInvariant} now follows since a scale-invariant function is by definition independent of the first factor in the product decomposition of the Hopf-Gaussian measure. 

%
%


\subsection{Moments of Edgelength, Expected Chordlengths, and Expected Gyradius} 
\label{sub:moments}

By Proposition~\ref{prop:HopfGaussianGenerated} the Hopf-Gaussian measure on $\sArm_d(n)$ for $d \in \{2,3\}$ is generated by the spherically symmetric pdf
\[
	g(\vecr) = \frac{e^{-\nicefrac{|\vecr|}{2}}}{4^{d-1} \pi |\vecr|}.
\]
This is certainly a bounded density with finite first, second, and third moments, so we can use Theorem~\ref{thm:asymptoticCurvature} to compute the asymptotic expected total curvature on $\sPol_d(n)$ with respect to the Hopf-Gaussian measure. To do so, we just need to know the first and second moments of edgelength. In fact, as we saw in the proof of Proposition~\ref{prop:HopfGaussianGenerated}, the edgelength distribution on $\sArm_d(n)$ is the chi-squared distribution with $2^{d-1}$ degrees of freedom, so the moments of edgelength are just the well-known moments of this distribution:

\begin{proposition}\label{prop:armMoments}
	The $p$th moment of edgelength on $\sArm_d(n)$ is given by
	\begin{equation*}
		E(|e_i|^p;\sArm_2(n),\operatorname{H}) = 2^p p!, \quad  E(|e_i|^p;\sArm_3(n),\operatorname{H}) = 2^p(p+1)!
	\end{equation*}
\end{proposition}

Note that the expected values of edgelength are
\begin{equation}\label{eq:1stMoment}
	E(|\vec{e}_i|;\sArm_2(n),\operatorname{H}) = 2, \quad E(|\vec{e}_i|;\sArm_3(n),\operatorname{H}) = 4
\end{equation}
and the expected squared edgelengths are
\begin{equation}\label{eq:2ndMoment}
 	\text{and} \quad E(|\vec{e}_i|^2;\sArm_2(n),\operatorname{H}) = 8, \quad E(|\vec{e}_i|^2;\sArm_3(n),\operatorname{H}) = 24.
\end{equation}

Using the above values for $m_1$ and $m_2$,  Theorem~\ref{thm:asymptoticCurvature} implies that the asymptotic expected total curvature on $\sPol_d(n)$ is given by
\[
	E(\kappa; \sPol_2(n), \operatorname{H}) \simeq \frac{\pi}{2}n + \frac{2}{\pi}, \quad  E(\kappa; \sPol_3(n), \operatorname{H}) \simeq \frac{\pi}{2}n + \frac{\pi}{4}.
\]
Since total curvature is scale-invariant, Theorem~\ref{thm:scaleInvariant} implies Corollary~\ref{cor:asymptoticTotalCurvature}, which says that for large $n$
\[
	E(\kappa; \Pol_2(n),\sigma) \simeq \frac{\pi}{2}n + \frac{2}{\pi}, \quad  E(\kappa; \Pol_3(n),\sigma) \simeq \frac{\pi}{2}n + \frac{\pi}{4}.
\]

Also, we can now compute the expected value of chordlength and radius of gyration for $\sArm_d(n)$ using results from \cite{cpam}.

\begin{corollary}\label{cor:armChordGyradius}
	The expected value of the squared length of a chord skipping $k$ edges on $\sArm_d(n)$ is
	\[
		E(\chord(k);\sArm_2(n),\operatorname{H}) = 8k, \quad E(\chord(k);\sArm_3(n),\operatorname{H}) = 24k.
	\]
	The expected squared radius of gyration for arms in $\sArm_d(n)$ is
	\[
		E(\gyradius;\sArm_2(n),\operatorname{H}) = \frac{4}{3}\frac{n(n+2)}{n+1}, \quad E(\gyradius;\sArm_3(n),\operatorname{H}) = 4\frac{n(n+2)}{n+1}.
	\]
\end{corollary}

\begin{proof}
	Since our measure on $\sArm_d(n)$ is invariant under rearrangement of edges, we can easily compute expected squared chord length and radius of gyration using Propositions~5.3,~6.3, and~6.5 from~\cite{cpam}. Those propositions imply that
	\begin{equation}\label{eq:chordArm}
		E(\chord(k);\sArm_d(n), \operatorname{H}) = k E(|\vec{e}_i|^2;\sArm_d(n), \operatorname{H})
	\end{equation}
	and
	\begin{equation}\label{eq:gyradiusArm}
		E(\gyradius;\sArm_d(n), \operatorname{H}) = \left(\frac{n(n+2)}{6(n+1)}\right) E(|\vec{e}_i|^2;\sArm_d(n), \operatorname{H}).
	\end{equation}
	Substituting the second moment of edgelength from \eqref{eq:2ndMoment} into \eqref{eq:chordArm} and \eqref{eq:gyradiusArm} yields the desired results.
\end{proof}




\section{Expected total curvature of polygons with the symmetric measure} 
\label{sec:length-2}

Corollary~\ref{cor:asymptoticTotalCurvature} gave the asymptotic expected total curvatures with respect to the symmetric measures as
	\[
		E(\kappa; \Pol_2(n),\sigma) \simeq \frac{\pi}{2}n + \frac{2}{\pi}, \quad E(\kappa; \Pol_3(n),\sigma) \simeq  \frac{\pi}{2}n + \frac{\pi}{4}.
	\]
Our aim in this section is to use the special properties of the Hopf-Gaussian measure to carry out the argument of Theorem~\ref{thm:asymptoticCurvature} for space polygons with no approximations. This will yield the following exact expectation which holds for any $n\geq 3$:

\begin{theorem}\label{thm:exact}
	The expected total curvature with respect to the symmetric measure on $\Pol_3(n)$ and the Hopf-Gaussian measure on~$\sPol_3(n)$ is given by
\begin{equation*}
E(\kappa; \Pol_3(n),\sigma) = E(\kappa; \sPol_3(n),\operatorname{H}) = \frac{\pi}{2}n + \frac{\pi}{4} \frac{2n}{2n-3}.
\end{equation*}
\end{theorem}

Since total curvature is scale-invariant, the first equality is a consequence of Theorem~\ref{thm:scaleInvariant}; proving the second equality is the main task of this section. Also, since all triangles have total curvature $2\pi$, Theorem~\ref{thm:exact} is trivially true for $n=3$. Therefore, throughout the rest of the section we will assume $n >3$.

We will shortly be evaluating many definite integrals involving the Bessel function $K_\nu(z)$; to do so we will repeatedly avail ourselves of the following:

\begin{lemma}
	\label{lem:mwRule}
	For real $\mu > |\nu|$ and $\alpha > 0$,
	\begin{equation}\label{eq:mwRule1}
		\int_0^\infty x^{\mu-1}e^{-\alpha x} K_\nu(\alpha x) \dx= \frac{\sqrt{\pi}}{2^\mu \beta^\mu} \frac{\Gamma(\mu+\nu)\Gamma(\mu-\nu)}{\Gamma(\mu+\nicefrac{1}{2})}
	\end{equation}
	and
	\begin{equation}\label{eq:mwRule2}
		\int_0^\infty x^{\mu-1} K_\nu(\alpha x) \dx = \frac{\sqrt{\pi}2^\nu}{\alpha^\mu}\frac{\Gamma(\mu-\nu)}{\mu+\nu}\frac{\Gamma(\nicefrac{\mu}{2}+\nicefrac{\nu}{2})}{\Gamma(\nicefrac{\mu}{2}+\nicefrac{1}{2})}.
	\end{equation}
\end{lemma}

\begin{proof}
	Both equations follow from the identity \cite[6.621(3)]{Gradshteyn:2007uy}
	\[
		\int_0^\infty x^{\mu-1} e^{-\alpha x} K_\nu (\beta x) \dx = \frac{\sqrt{\pi}(2 \beta)^\nu}{(\alpha + \beta)^{\mu+\nu}} \frac{\Gamma(\mu+\nu)\Gamma(\mu-\nu)}{\Gamma(\mu+\nicefrac{1}{2})}     \ _2F_1\left(\mu+\nu,\nu+\frac{1}{2}; \mu + \frac{1}{2}; \frac{\alpha - \beta}{\alpha+\beta}\right),
	\]
	which holds for any complex numbers $\mu, \nu, \alpha, \beta$ with $\Re (\mu) > |\Re (\nu)|$ and $\Re(\alpha + \beta) > 0$. The function $_2F_1$ is Gauss's hypergeometric function.
	
	To get \eqref{eq:mwRule1}, we simply use the fact that $_2F_1(a,b;c;0) = 1$ for any $a,b,c$; to get~\eqref{eq:mwRule2} we use Kummer's identity: $_2F_1(a,b;a-b+1;-1) = \frac{\Gamma(a-b+1)\Gamma(\nicefrac{a}{2}+1)}{\Gamma(a+1)\Gamma(\nicefrac{a}{2}-b+1)}$.
\end{proof}

\subsection{The Sum of $k$ Edges in $\sArm_3(n)$}

Since the goal is to carry out the strategy from Section~\ref{sec:total-curvature} with no approximations, we first need to explicitly determine the Green's function $G_k$:

\begin{proposition}  \label{prop:failureToClose}
The probability distribution of the vector $\vec{r}$ joining the ends of a $k$-edge sub-arm of an arm in $\sArm_3(n)$ with the Hopf-Gaussian measure is spherically symmetric in $\R^3$ and given by the following explicit formula:
\[
	G_k(\vec{r}) \dVol_{\vec{r}} = \frac{r^{k-\nicefrac{3}{2}} K_{k-\nicefrac{3}{2}}\left(\nicefrac{r}{2}\right)}{2^{2k+2} \pi^{\nicefrac{3}{2}}\Gamma(k)} \dVol_{\vec{r}},
\]
where $r = |\vec{r}|$.
\end{proposition}

\begin{proof}
Suppose $\vec{q} = (q_1, \ldots , q_n) \in \Q^k$ is sampled from the standard Gaussian distribution. Writing $q_i = a_i + b_i \I + c_i\J + d_k\K$ for each $i = 1, \ldots , k$, the failure-to-close vector for the $k$-edge arm $\Hopf(\vec{q}) \in \sArm_3(n)$ is 
	\[
		\sum \Hopf(q_i) = \sum \Hopf(a_i + b_i \I+c_i\J + d_k\K)  = \sum (a_i^2 + b_i^2 - c_i^2 - d_i^2, 2b_ic_i - 2a_id_i, 2a_ic_i+2b_id_i).
	\]
	Again, the fact that this vector follows a spherically symmetric distribution is a consequence of the fact that the Hopf map is $SU(2)$-equivariant.
	
	Since the $a_i, b_i, c_i, d_i$ are chosen from standard (real) Gaussian distributions, the distribution of the projection of the failure-to-close vector onto the first coordinate is clearly the difference of two chi-squared distributions, each with $2k$ degrees of freedom. This is known to have a pdf given by a Bessel function distribution~\cite[Chapter 12, Section 4.4]{Johnson:1994tg} in the form
	\begin{equation}\label{eq:projPDF}
		f(y) = \frac{|y|^{k-\nicefrac{1}{2}}}{4^k \sqrt{\pi}\Gamma(k)} K_{k-\nicefrac{1}{2}}\left(\nicefrac{|y|}{2}\right).
	\end{equation}
It is worth noting that this pdf was proved by McLeish~\cite{McLeish:1982fg} to be the pdf of a product of a gamma $(k,2)$ variable and an independent standard normal. McLeish also works out the moments and cumulants of the distribution.	
	
	Next, we will use a result of Lord~\cite{Lord:1954wh} which relates the pdf of the projection of a spherically symmetric distribution to the pdf of the full distribution.
	
	\begin{lemma}[{\cite[Eq.\ (29)]{Lord:1954wh}}]\label{lem:Lord}
		Suppose $p(\vec{r})$ is a spherically symmetric distribution on $\R^3$ and that the projection of $p(\vec{r})$ to any radial line through the origin has the pdf $p_1(r)$, where $r = |\vec{r}|$. Then $p(\vec{r})$ is given by 
		\[
			p(\vec{r}) = -\frac{1}{2\pi r} p_1'(r) \dVol_{\vec{r}}.
		\]
	\end{lemma}
	
	Using the pdf of the projection given in \eqref{eq:projPDF}, Lemma~\ref{lem:Lord} implies that the failure-to-close distribution is
	\[
		G_k(\vec{r}) \dVol_{\vec{r}} = -\frac{1}{2\pi r} \frac{d}{dr} \left(\frac{r^{k-\nicefrac{1}{2}}}{4^k \sqrt{\pi}\Gamma(k)} K_{k-\nicefrac{1}{2}}\left(\nicefrac{r}{2}\right)\right) \dVol_{\vec{r}}
	\]
	This can be re-written using the derivative identity $K_\nu'(u) = -\frac{1}{2}(K_{\nu-1}(u)+K_{\nu+1}(u))$ and the recurrence relation $K_\nu(u) = K_{\nu-2}(u) + \frac{2(\nu-1)}{u}K_{\nu-1}(u)$ (cf. \cite[10.29.1]{nist} for both) to get the desired expression
	\[
		G_k(\vec{r}) \dVol_{\vec{r}} = \frac{r^{k-\nicefrac{3}{2}} K_{k-\nicefrac{3}{2}}\left(\nicefrac{r}{2}\right)}{2^{2k+2} \pi^{\nicefrac{3}{2}}\Gamma(k)} \dVol_{\vec{r}}
	\]
	where $r = |\vec{r}|$. 	%
\end{proof}

Using the identity $K_{-\nicefrac{1}{2}}(z) = \frac{\sqrt{\pi}}{\sqrt{2z}} e^{-z}$ (cf. \cite[10.39.2]{nist}) and specializing Proposition~\ref{prop:failureToClose} to the case $k=1$, we see that the distribution of an edge in a Hopf-Gaussian arm is
\[
	G_1(\vec{r}) \dVol_{\vec{r}} = \frac{ e^{-\nicefrac{r}{2}}}{16 \pi r} \dVol_{\vec{r}}.
\]
Note that this is, as it should be, the same distribution for edges that we found in Proposition~\ref{prop:HopfGaussianGenerated}.

\subsection{The pdf of Edges in $\sPol_3(n)$}
\label{sub:pdfs}

\begin{proposition}
\label{prop:cn}
The codimension 3 Hausdorff measure of $\sPol_3(n)$ in $\sArm_3(n)$ is the value $G_n(\vec{0})$, which is given by
\begin{equation*}
C_n = \frac{\Gamma(n-\nicefrac{3}{2})}{64 \sqrt{\pi} \Gamma(n)}
\end{equation*}
\end{proposition}

\begin{proof}
We saw in Proposition~\ref{prop:failureToClose} that 
\begin{equation}
G_n(\vec{r}) = \frac{2^{-2n-2} r^{n - \nicefrac{3}{2}} K_{n - \nicefrac{3}{2}}(\nicefrac{r}{2})}{\pi^{\nicefrac{3}{2}} \Gamma(n)}.
\label{eq:gn}
\end{equation}
We must be careful evaluating this formula at $r = 0$, since the Bessel function $K_{n - \nicefrac{3}{2}}$ has a pole at~$0$. To rewrite~\eqref{eq:gn} in a form which allows us to easily evaluate at $r = 0$, we use the general formula for Bessel functions of half-integer order from \cite[p.\ 80, formula (12)]{Watson:1995ui}:
\begin{equation*}
K_{n + \nicefrac{1}{2}}(z) = \left( \frac{\pi}{2z} \right)^{\nicefrac{1}{2}} e^{-z} \sum_{i=0}^n \frac{(n+i)!}{i! (n-i)! (2z)^i}.
\end{equation*}
Writing $n - \nicefrac{3}{2} = (n-2) + \nicefrac{1}{2}$, we see that $G_n(\vec{r})$ simplifies to
\begin{equation*}
G_n(\vec{r}) = \frac{2^{-2 (n+1)} e^{-\nicefrac{r}{2}}}{\pi  \Gamma (n)}
   \sum_{i=0}^{n-2} \frac{(i+(n-2))!} {i!((n-2) - i)!} r^{(n-2)-i}.
\end{equation*}
In this form, it is clear that the only term in the sum which is nonzero at $r=0$ is the $i = n-2$ term and
\begin{equation*}
G_n(\vec{0}) = \frac{2^{-2(n+1)} (2n-4)!}{\pi \Gamma(n) (n-2)!} = \frac{2^{-2(n+1)} \Gamma(2n-3)}{\pi \Gamma(n) \Gamma(n-1)}.
\end{equation*}
We can simplify this a bit further using the duplication formula for gamma functions~\cite[5.5.5]{nist}
$\Gamma(2z) = \pi^{-\nicefrac{1}{2}} 2^{2z-1} \Gamma(z) \Gamma(z+\nicefrac{1}{2})$ to get 
\begin{equation}
C_n = G_n(\vec{0}) = \frac{\Gamma(n - \nicefrac{3}{2})}{64 \sqrt{\pi} \Gamma(n)}.
\end{equation}
as desired. 
\end{proof}
We can use the pdf for sums of edges in $\sArm_3(n)$ to write down the pdfs for single edges and pairs of edges in $\sPol_3(n)$. 

\begin{proposition}\label{prop:singleEdgePolPDF}
The pdf of a single edge $\vec{e}_i$ in $\sPol_3(n)$ with respect to $\dVol_{\vec{e}_i}$ is spherically symmetric on $\R^3$ and given by the following function of $r_i = |\vec{e}_i|$:
\[
	P(\vec{e}_i) = \frac{n-1}{2^{2n-2}\pi \Gamma(n-\nicefrac{3}{2})} e^{-\nicefrac{r_i}{2}} r_i^{n-\nicefrac{7}{2}} K_{n-\nicefrac{5}{2}} \left(\nicefrac{r_i}{2}\right).
\]
\end{proposition}

\begin{proof}
	The pdf of a single edge is just
	\[
		P(\vec{e}_i) = \frac{1}{C_n} g(\vec{e}_i) G_{n-1}(-\vec{e}_i); 
	\]
	i.e., the probability of the $i$th edge being $\vec{e}_i$ and the remaining $n-1$ edges summing to $-\vec{e}_i$, conditioned on the assumption that all $n$ edges sum to zero. Since $g(\vec{e}_i) = G_1(\vec{e}_i)$, we can use Proposition~\ref{prop:failureToClose} and Proposition~\ref{prop:cn} to arrive at the stated expression.
\end{proof}

\begin{corollary}
The moments of edgelength for polygons in $\sPol_3(n)$ are 
\[
	E(|\vec{e}_i|^p;\sPol_3(n),\operatorname{H}) = \frac{(n-1)\Gamma(2n+p-3)}{2 \Gamma(2n-4)}\Beta(p+2,n-2).
\]

and hence the expectation of squared chordlength and radius of gyration are
\begin{align*}
	E(\chord(k);\sPol_3(n),\operatorname{H}) & = \left(\frac{n-k}{n}\right)\frac{12k(2n-3)}{n+1} \\ 
	E(\gyradius;\sPol_3(n),\operatorname{H}) & = \left(\frac{n-1}{n}\right) (2n-3).
\end{align*}
\end{corollary}

\begin{proof}
	Using the pdf from Proposition~\ref{prop:singleEdgePolPDF}, the expected value of $|\vec{e}_i|^p = r_i^p$ is
	\[
		\int_{\R^3} r_i^p P(\vec{e}_i) \dVol_{\vec{e}_i} = \int_0^\infty \int_0^\pi \int_0^{2\pi} r_i^p P(\vec{e}_i) \, r_i^2 \sin \phi \dtheta \dphi \dr_i.
	\]
	Writing out $P(\vec{e}_i)$ and integrating with respect to $\theta$ and $\phi$ gives the $p$th moment
	\begin{multline*}
		\frac{n-1}{2^{2n-4}\Gamma(n-\nicefrac{3}{2})} \int_0^\infty e^{-\nicefrac{r_i}{2}} r_i^{n+p-\nicefrac{3}{2}} K_{n-\nicefrac{5}{2}}\left(\nicefrac{r_i}{2}\right) \dr_i\\
		= \frac{n-1}{2^{2n-4}\Gamma(n-\nicefrac{3}{2})} \frac{\sqrt{\pi}\Gamma(p+2)\Gamma(2n+p-3)}{\Gamma(n+p)}	\\
		= \frac{(n-1)\Gamma(2n+p-3)}{2 \Gamma(2n-4)} \Beta(p+2,n-2)	%
	\end{multline*}
	using Lemma~\ref{lem:mwRule} and the duplication rule for the gamma function.	
			
	
	Since the measure on $\sPol_3(n)$ is invariant under rearrangement of edges, we can use Propositions~6.3 and~6.5 from \cite{cpam} as in Corollary~\ref{cor:armChordGyradius}. Specifically,
	\[
		E(\chord(k);\sPol_3(n),\operatorname{H}) = \left(\frac{n-k}{n-1}\right) k E(|\vec{e}_i|^2;\sPol_3(n), \operatorname{H})
	\]
	and
	\[
		E(\gyradius;\sPol_3(n),\operatorname{H}) = \left(\frac{n+1}{12}\right) E(|\vec{e}_i|^2;\sPol_3(n), \operatorname{H}).
	\]
	
	By the first part of the proposition we have that
	\[
		E(|\vec{e}_i|^2;\sPol_3(n),\operatorname{H}) = \frac{12(n-1)(2n-3)}{n(n+1)},
	\]
	so the given formulas for $E(\chord(k);\sPol_3(n),\operatorname{H})$ and $E(\gyradius;\sPol_3(n),\operatorname{H})$ are immediate.
\end{proof}

\begin{proposition}
The probability distribution of a pair of edges $\vec{e_1}$ and $\vec{e_2}$ in $\sPol_3(n)$ is invariant under the diagonal action of $SO(3)$ on the pair of edges and invariant under rotations which fix one edge. Hence, it depends only on the lengths $r_1$ and $r_2$ of the two edges and the angle $\theta$ between them. It is given by the formula
\begin{equation}
P(r_1,r_2,\theta) \dr_1 \dr_2 \dtheta = \frac{\Gamma(n)}{4 \sqrt{\pi} \Gamma(2n-4)} r_1 r_2 e^{-\frac{1}{2}(r_1 + r_2)} z^{n - \frac{7}{2}}
K_{n - \frac{7}{2}}\left(\frac{z}{2} \right) \sin \theta \dr_1 \dr_2 \dtheta,
\end{equation}
where $z = |\vec{e}_1 + \vec{e}_2| = \sqrt{r_1^2 + r_2^2 + 2 r_1 r_2 \cos \theta}$.
\label{prop:pn}
\end{proposition}

\begin{proof}
As we saw above, the general form for the probability distribution of a pair of edges in a closed polygon in $\sPol_3(n)$ is
\begin{equation}
P(\vec{e_1},\vec{e_2}) \dVol_{\vec{e}_1} \dVol_{\vec{e}_2} = \frac{g(\vec{e_1}) g(\vec{e_2}) G_{n-2}(-\vec{e_1} - \vec{e_2})}{C_n} \dVol_{\vec{e}_1} \dVol_{\vec{e}_2},
\end{equation}
that is, the probability of the first two edges being $\vec{e_1}$ and $\vec{e_2}$ and the remaining edges summing to $-\vec{e_1} - \vec{e_2}$ conditioned on the assumption that all $n$ edges sum to zero. We computed $G_k(\vec{r})$ in Proposition~\ref{prop:failureToClose} and $C_n$ in Proposition~\ref{prop:cn}. Using the formula for $g(\vec{e})$ from Proposition~\ref{prop:HopfGaussianGenerated}, we get 
\begin{equation*}
P(\vec{e_1},\vec{e_2}) \dVol_{\vec{e}_1} \dVol_{\vec{e}_2} = \frac{\Gamma(n)}{32 \pi^{\nicefrac{5}{2}} r_1 r_2 \Gamma(2n-4)} e^{-\nicefrac{1}{2}(r_1 + r_2)} z^{n - \nicefrac{7}{2}}
K_{n - \nicefrac{7}{2}}\left(\nicefrac{z}{2} \right) \dVol_{\vec{e}_1} \dVol_{\vec{e}_2},
\end{equation*}
where again $z = |\vec{e}_1 + \vec{e}_2| = \sqrt{r_1^2 + r_2^2 + 2 r_1 r_2 \cos \theta}$. We can rewrite $\vec{e_1}$ in spherical coordinates $\vec{e_1} = (r_1,\phi_1,\theta_1)$. For each $\vec{e_1}$, we can fix $\vec{e_1}$ as the $z$-axis of spherical coordinates for ${\vec{e_2} = (r_2,\phi_2,\theta_2)}$. Here $\theta_2$ is equal to $\theta$, the angle between $\vec{e_1}$ and $\vec{e_2}$. Observing that we can integrate out $\phi_1$ and $\phi_2$ immediately to get a factor of $4 \pi^2$ and $\theta_1$ to get a factor of $2$, and recording the volume form in these coordinates as $r_1^2 r_2^2 \sin \theta \dr_1 \dr_2 \dtheta$ gives us the formula in the statement of the proposition.
\end{proof}

\begin{corollary}\label{cor:pxyz}
The pairwise distribution of edges $\vec{e}_1$ and $\vec{e}_2$ in $\sPol_3(n)$ for $n > 3$ may also be written more simply in terms of the variables $x = (r_1 + r_2)/2$, $y = (r_1 - r_2)/2$, and $z=|\vec{e}_1 + \vec{e}_2|$. In these variables, the probability distribution is given by:
\begin{equation}
 P(x,y,z) \dx \dy \dz = \frac{\Gamma (n)}{2 \sqrt{\pi }\Gamma (2 n-4)} e^{-x} z^{n-\nicefrac{5}{2}} 	K_{n-\nicefrac{7}{2}}\left(\nicefrac{
	   z}{2}\right) \dx \dy \dz.
\label{eq:pxyz}
\end{equation}
\end{corollary}

\begin{proof}
Computing the Jacobian of the map $(r_1,r_2,\theta) \mapsto (x,y,z)$, we see that its inverse determinant is 
\begin{equation*}
|\mathcal{J}^{-1}| = \frac{2 \csc (\theta )
   \sqrt{r_1^2+ r_2^2 + 2 r_1
   r_2 \cos \theta}}{r_1 r_2}
\end{equation*}
We then multiply the function in Proposition~\ref{prop:pn} by this determinant and substitute to obtain the statement of the Corollary.
\end{proof}

We now explicitly check that the total integral of the pairwise pdf is equal to 1. This computation will serve as a warm-up for the more difficult definite integrals ahead. 
%

Since $y$ does not appear in the pdf $P(x,y,z)$ in~\eqref{eq:pxyz}, we will integrate with respect to $y$ first. The triangle inequality states that $r_1 \leq r_2 + z$, so $y \leq \nicefrac{z}{2}$, and similarly $r_2 \leq r_1 + z$, so $-\nicefrac{z}{2} \leq y$. We will next integrate by $x$, since $x$ does not appear in a Bessel function. There is no upper bound on $x = (r_1 + r_2)/2$, but since $r_1 + r_2 \geq z$, we know $x \geq \nicefrac{z}{2}$. We will integrate with respect to $z$ last, and here the limits are simply $0$ and $\infty$. Thus, we are trying to show that 
\begin{equation}
\frac{\Gamma(n)}{2 \sqrt{\pi} \Gamma(2n - 4)} \int_0^\infty \int_{\nicefrac{z}{2}}^\infty \int_{-\nicefrac{z}{2}}^{\nicefrac{z}{2}} e^{-x} z^{n-\nicefrac{5}{2}}
   K_{n-\nicefrac{7}{2}}\left(\nicefrac{z}{2}\right)  \dy  \dx  \dz
= 1.
\end{equation}
The $y$ and $x$ integrals are simple, and leave us with
\begin{equation}
\frac{\Gamma(n)}{2\sqrt{\pi} \Gamma(2n-4)} \int_0^\infty e^{-\nicefrac{z}{2}} z^{n-\nicefrac{3}{2}}
   K_{n-\nicefrac{7}{2}}\left(\nicefrac{
   z}{2}\right) \dz = 1
\end{equation}
by Lemma~\ref{lem:mwRule}. %
This check gives us confidence that our pairwise pdf is correct so far.

\subsection{Expected Turning Angles and Total Curvature}

\begin{proposition}
The expected value of the turning angle $\theta$ for a single pair of edges in $\sPol_3(n)$ is given by the formula
\begin{equation}
E(\theta) = \frac{\pi}{2} + \frac{\pi}{4} \frac{2}{2n - 3}
\end{equation}
\label{prop:turningangle}
\end{proposition}

\begin{proof}
Our overall strategy will be to use the formula for $P(x,y,z)$ from Corollary~\ref{cor:pxyz} to write this expected value as 
\begin{equation*}
E(\theta) = \int_0^\infty \int_{\nicefrac{z}{2}}^\infty \int_{-\nicefrac{z}{2}}^{\nicefrac{z}{2}} \theta(x,y,z) P(x,y,z) \dy \dx \dz.
\end{equation*}

We begin by writing the turning angle $\theta$ in terms of the $x$, $y$, and $z$ variables as
\begin{equation*}
\theta(x,y,z) = \arccos \left( \frac{z^2 - 2 (x^2 + y^2)}{2(x^2 - y^2)} \right).
\end{equation*}
Since $P(x,y,z)$ does not depend on $y$, the first integral is accomplished by integrating this function with respect to $y$. Integrating by parts with $dv = 1$, we get
\begin{equation*}
\int_{-\nicefrac{z}{2}}^{\nicefrac{z}{2}} \theta(x,y,z) \dy = \pi z + \int_{-\nicefrac{z}{2}}^{\nicefrac{z}{2}} \frac{2y^2 \sqrt{4x^2-z^2}}{(y^2-x^2)\sqrt{z^2-4y^2}} \dy.
\end{equation*}
Now, making the trig substitution $y = \frac{z}{2} \sin \psi$ yields
\[
	\pi z + \sqrt{4x^2-z^2} \int_{-\nicefrac{\pi}{2}}^{\nicefrac{\pi}{2}} \frac{\sin^2\psi}{\sin^2\psi-\frac{4x^2}{z^2}} \,\mathrm{d}\psi = \pi z + \sqrt{4x^2 - z^2} \int_{-\nicefrac{\pi}{2}}^{\nicefrac{\pi}{2}} \left(1 - \frac{8x^2}{8x^2-z^2+z^2\cos 2\psi}\right) \mathrm{d}\psi
\]
by using the identity $\sin^2\psi = \frac{1-\cos 2\psi}{2}$. This is simply
\[
	\pi z + \pi \sqrt{4x^2-z^2} - 2\pi x
\]
using the indefinite integral identity~\cite[2.558(4)]{Gradshteyn:2007uy}
\[
	\int \frac{\mathrm{d}\psi}{a + b\cos \psi} = \frac{2}{\sqrt{a^2-b^2}} \arctan \frac{(a-b) \tan \nicefrac{\psi}{2}}{\sqrt{a^2-b^2}}.
\]

Since the function $P(x,y,z)$ is in the form $e^{-x} f(z)$, we must now do the pair of integrals 
\begin{equation*}
\pi \int_{\frac{z}{2}}^\infty e^{-x} (z - 2x) \dx + \pi \int_{\frac{z}{2}}^\infty e^{-x} \sqrt{4x^2-z^2} \dx .
\end{equation*}
The first integral is simple and has the value $-2\pi e^{-\nicefrac{z}{2}}$. To do the second integral, we make the change of variables $x = (\nicefrac{z}{2}) t$ to get
\begin{equation*}
\int_{\nicefrac{z}{2}}^\infty e^{-x} \sqrt{4x^2-z^2} \dx = \frac{z^2}{2} \int_{1}^\infty e^{-t \frac{z}{2}} (t^2 - 1)^{\frac{1}{2}} \dt = z K_1\left(\nicefrac{z}{2}\right),
\end{equation*}
recognizing the last integral as a form of the integral representation for Bessel $K_\nu$~\cite[10.32.8]{nist} that we have used before. We have now shown that 
\begin{align}
\nonumber E(\theta) & = \frac{\sqrt{\pi} \Gamma(n)}{2 \Gamma(2n - 4)} 
\left( 
\int_0^\infty \hspace{-.1in}-2 e^{-\nicefrac{z}{2}} z^{n - \nicefrac{5}{2}} K_{n - \nicefrac{7}{2}} \left( \nicefrac{z}{2} \right) \dz + \int_0^\infty z^{n - \nicefrac{3}{2}} K_1 \left( \nicefrac{z}{2} \right) K_{n - \nicefrac{7}{2}} \left( \nicefrac{z}{2} \right) \dz 
\right)\\
\label{eq:rob1} & = - \frac{n-1}{2n-5}\pi + \frac{\sqrt{\pi} \Gamma(n)}{2 \Gamma(2n - 4)} \int_0^\infty z^{n - \nicefrac{3}{2}} K_1 \left( \nicefrac{z}{2} \right) K_{n - \nicefrac{7}{2}} \left( \nicefrac{z}{2} \right) \dz
\end{align}
where the integral is, as usual, computed using Lemma~\ref{lem:mwRule}.

The remaining integral in \eqref{eq:rob1} is more interesting, as it involves a product of Bessel functions. Making the substitution $z = 2u$, we can then use the Nicholson integral representation for the product of Bessel functions~\cite[10.32.17]{nist} to rewrite the integral as 
\begin{multline*}
\frac{\sqrt{\pi} \Gamma(n)}{2 \Gamma(2n - 4)} \int_{0}^\infty
z^{n - \nicefrac{3}{2}} K_1 \left( \nicefrac{z}{2} \right) K_{n - \nicefrac{7}{2}} \left( \nicefrac{z}{2} \right) \dz 
= \\
\frac{\sqrt{\pi }
   2^{n-\nicefrac{1}{2}} \Gamma (n)}{\Gamma (2 n-4)}
   \int_0^\infty
   \cosh \left( (\nicefrac{9}{2}-n ) t\right)
   \left(\int_0^\infty
   u^{n-\nicefrac{3}{2}} K_{n-\nicefrac{5}{2}}(2 u \cosh t )
   \du \right)
   \dt.
\end{multline*}
Again, the inner integral is a power of $z$ multiplied by a Bessel function of $\alpha z$ (here $\alpha = 2 \cosh t$) and can hence be evaluated using Lemma~\ref{lem:mwRule}. We now have the integral
\begin{equation}
\pi 2^{\nicefrac{5}{2}-n} (n-2)(n-1) 
   \int_0^\infty \cosh
   \left( (\nicefrac{9}{2}-n) \, t \right)
   \text{sech}^{n-\nicefrac{1}{2}} t \dt   .
\label{eq:rob2integral}
\end{equation}
We will integrate this using the general integration formula
\begin{equation}
\int_0^\infty \cosh \alpha x \sech^\beta x \dx = \sqrt{\frac{\pi}{2}} \, \frac{\Gamma(\beta+\alpha) \Gamma(\beta-\alpha) P_{\alpha-\nicefrac{1}{2}}^{\nicefrac{1}{2} - \beta}(0)}{\Gamma(\beta)},
\end{equation}
where $P_x^y$ is the associated Legendre function. This is valid when $\beta - \alpha > 0$ and $\beta + \alpha > 0$. It is a specialization of~\cite[3.5.17]{Gradshteyn:2007uy}. We set $\alpha = \nicefrac{9}{2} - n$ and $\beta = n - \nicefrac{1}{2}$ and see that $\beta - \alpha = 2n - 5$ (which is positive since we have assumed $n > 3$) and $\beta + \alpha = 4$. Applying the formula shows that the integral of~\eqref{eq:rob2integral} is equal to  
\begin{equation}
\frac{3 \pi^{\nicefrac{3}{2}} 2^{3-n} (n-2)
   (n-1) P_{4-n}^{1-n}(0) \Gamma
   (2 n-5)}{\Gamma
   \left(n-\nicefrac{1}{2}\right)}  =
   \frac{4 (n-2) (n-1)}{(2n-5) (2 n-3)} \pi,
\label{eq:rob2}
\end{equation} 
where we have used the general formula for the value at zero of associated Legendre functions given in~\cite[14.5.1]{nist} and the duplication formula for gamma functions to simplify the form on the left hand side.
Combining~\eqref{eq:rob1} and~\eqref{eq:rob2}, we see that 
\begin{equation}
E(\theta) = \frac{n-1}{2n-3} \pi = \frac{\pi}{2} + \frac{\pi}{4n - 6} = \frac{\pi}{2} + \frac{\pi}{4} \frac{2}{2n - 3},
\end{equation}
as desired.
\end{proof}

Theorem~\ref{thm:exact} now follows by multiplying by $n$. We get an interesting corollary of this theorem for $n=6$ and $n=7$; since the expected value of total curvature is less than $4\pi$, some hexagons and heptagons must have total curvature less than $4\pi$ and hence be unknotted by the F\'ary-Milnor theorem \cite{MR12:273c}.

\begin{corollary}
\label{cor:unknots}
If we measure volume using the symmetric measure on polygon space, at least $\nicefrac{1}{3}$ of the polygons in $\Pol_3(6)$ and $\nicefrac{1}{11}$ of the polygons in $\Pol_3(7)$ are unknotted.
\end{corollary}

\begin{proof}
Let $x$ be the fraction of polygons in $\Pol_3(n)$ with total curvature greater than $4 \pi$. By the F\'ary-Milnor theorem, these are the only polygons which may be knotted. We know that any closed polygon has total curvature at least $2\pi$, so the expected value of total curvature satisfies
\begin{equation*}
E(\kappa;\Pol_3(n),\sigma) > 4 \pi x + 2 \pi (1 - x).
\end{equation*}
Solving for $x$ and using Theorem~\ref{thm:exact}, we see that 
\begin{equation*}
x < \frac{(n-2)(n-3)}{2(2n-3)}.
\end{equation*}
For $n > 7$, this bound is not an improvement on the trivial bound $x \leq 1$, but for $n = 6$, we get $x < \nicefrac{2}{3}$ and for $n=7$, we get $x < \nicefrac{10}{11}$, as desired.
\end{proof}

If instead we had let $x$ be the fraction of polygons in $\Pol_3(n)$ with total curvature greater than $2\pi B$, then $x < \frac{(n-2)(n-3)}{(B-1)2(2n-3)}$ and this gives constraints on the fraction of knots with bridge number $B$. For example, when $n=8$, $9$, $10$, or $11$ this gives us some information on the fraction of 3-bridge knots since there are 3-bridge knots (e.g. $8_{19}$) with stick number 8.

\section{New Numerical Methods for Random Polygons}
\label{sec:numerical}

In~\cite{cpam}, we gave a fast sampling algorithm for random polygons in $\Pol_3(n)$ which is guaranteed to sample directly from the symmetric probability measure on this space given a supply of normal random variates. The ensembles of polygons generated by this algorithm are as good as the underlying ensembles of normals, so the quality of this sampling algorithm cannot be improved. However, when one is computing the expected value of a geometric functional on polygon space such as total curvature, averaging over a large ensemble of sample polygons is simply Monte Carlo integration over a very high-dimensional space. The numerical accuracy of this method is necessarily limited. 

The framework above yields a much better method for computing expected values of functions like total curvature which are both scale-invariant and sums of quantities that are defined locally on a given polygon: integrate directly against the pdf for a finite collection of edges in a closed $n$-gon in $\sPol_3(n)$. Carrying out this method is not trivial if the integrand, like turning angle, has a singularity in most coordinate systems. However, the results are worth it. Figure~\ref{fig:curvature data} shows the curvature ``surplus'' term of total curvature minus $\frac{\pi}{2}n$ plotted with data from sampling and the number of correct digits obtained by sampling and by numerical integration.

\begin{figure}[ht]
\hphantom{.}
\hfill
\begin{overpic}[height=1.75in]{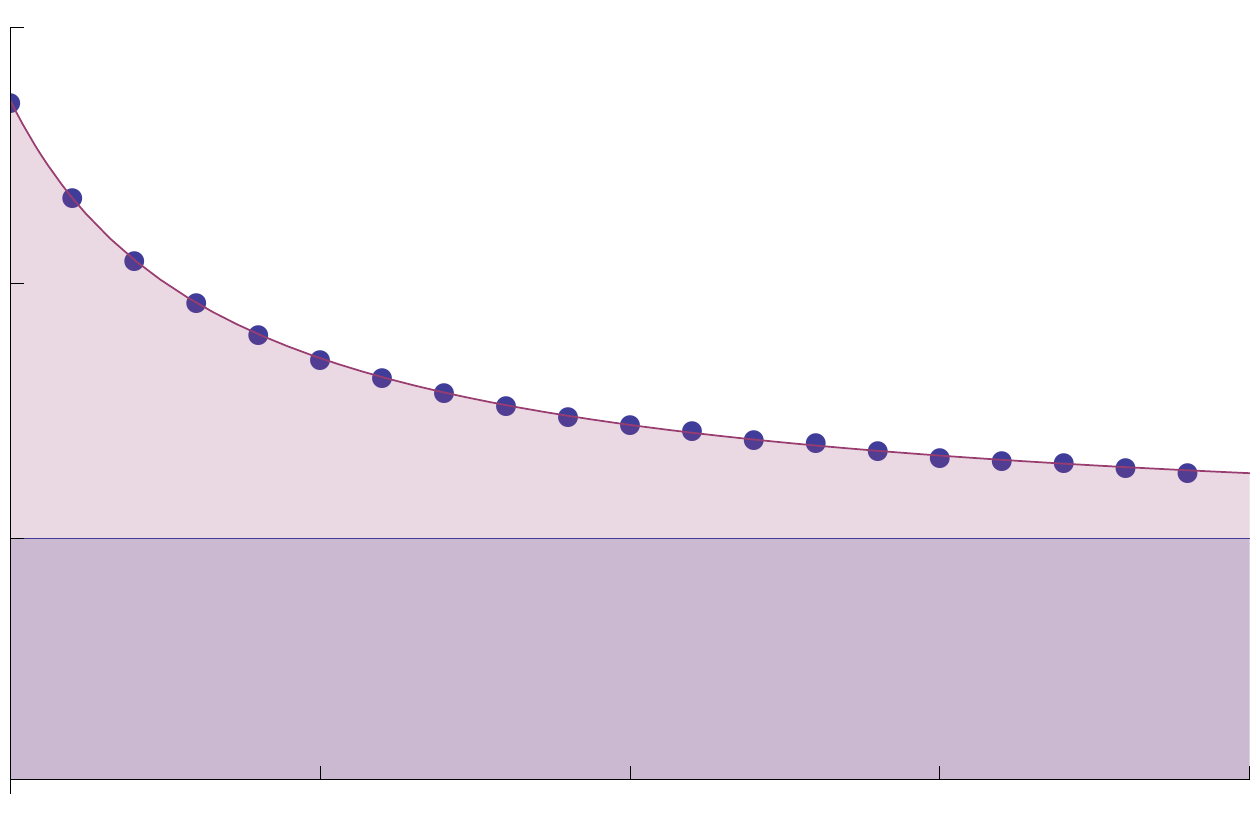}
\put(32,-6){Number of edges $n$}
\put(-3,66){Total Curvature Surplus}
\put(-5,22){$\frac{\pi}{4}$}
\put(-6.5,42.5){$\frac{5\pi}{16}$}
\put(23,-0.5){\footnotesize $10$}
\put(48,-0.5){\footnotesize $15$}
\put(72.5,-0.5){\footnotesize $20$}
\end{overpic}
\hfill
\begin{overpic}[height=1.75in]{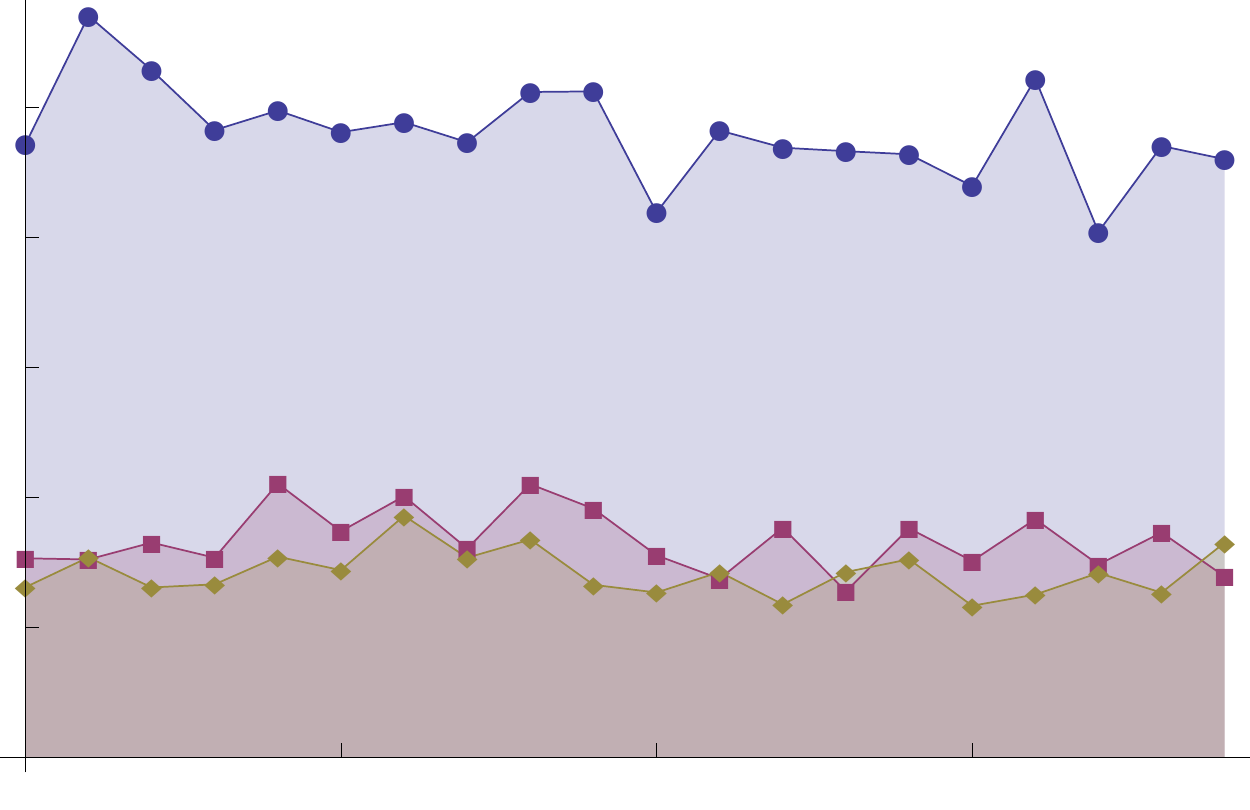}
\put(32,-6){Number of edges $n$}
\put(-10,66){Correct Digits}
\put(-2,13){\footnotesize $2$}
\put(-2,24){\footnotesize $4$}
\put(-2,34.5){\footnotesize $6$}
\put(-2,44){\footnotesize $8$}
\put(-4.5,55){\footnotesize $10$}
\put(25,-0.5){\footnotesize $10$}
\put(50,-0.5){\footnotesize $15$}
\put(75,-0.5){\footnotesize $20$}

\end{overpic}
\hfill
\hphantom{.}
\vspace{0.05in}

\caption{On the left, we see the average total curvature surplus $\kappa -  \frac{\pi}{2}n$ plotted for ensembles of 5 million polygons with 5 to 25 edges generated by the direct sampling algorithm of~\cite{cpam} (dots), plotted together with the exact formula for this surplus of $\frac{\pi}{4}\frac{2n}{2n-3}$ given by Theorem~\ref{thm:exact} (curve) and the asymptotic value of this surplus of $\frac{\pi}{4}$ given by Theorem~\ref{thm:asymptoticCurvature} (line). The fact that these three computations are in agreement serves as a useful check on our work above. On the right, we see the number of correct digits in the calculation of average total curvature by careful numerical integration of turning angle against the pairwise pdf of Corollary~\ref{cor:pxyz} (top line, about 10 correct digits), averaging over ensembles of 5 million polygons (middle line, about 4 correct digits), and averaging over ensembles of 1 million polygons (bottom line, about 3 correct digits). We can see that we obtain significantly better results by numerical integration.
\label{fig:curvature data}
}
\end{figure}
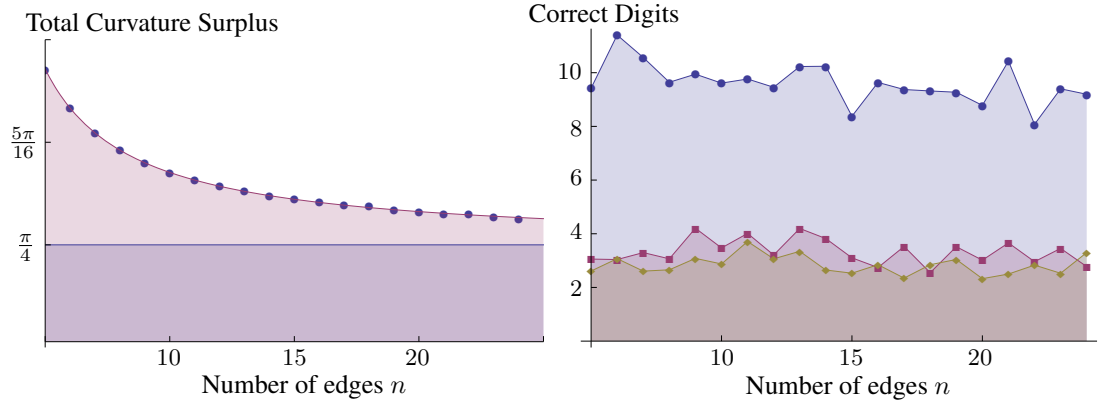

\section{Future Directions}

It is clear that the methods above have many interesting applications. For instance, we can hope to compute the expected total curvature for plane polygons as well. In this case Theorem~\ref{thm:asymptoticCurvature} shows that the expected curvature surplus is $\nicefrac{2}{\pi}$. However, the integral seems somewhat forbidding and we do not have an explicit conjecture for a closed form. 

More promising is the direction of extending our results to other functionals on space polygons. Since Theorem~\ref{thm:scaleInvariant} tells us that the expected value of \emph{any} scale invariant functional over the space of Hopf-Gaussian polygons is equal to the expected value over polygons with the symmetric measure, we can expect to compute a number of other interesting expectations this way. For example, the argument of Section~\ref{sec:total-curvature} can certainly be generalized to predict an asymptotic expected total torsion of $\frac{\pi}{2}n - \frac{\pi}{4}$. It would be interesting to find the exact expectation of total torsion. It seems possible to apply these methods to average crossing number as well, which is certainly a topic for further investigation!

In Corollary~\ref{cor:unknots}, we used our expectation for total curvature with respect to the symmetric measure to bound the fraction of unknotted fixed-length hexagons below by $\nicefrac{1}{3}$ and the fraction of unknotted fixed-length heptagons below by $\nicefrac{1}{11}$. Such bounds are clearly too small: numerically sampling ensembles of 5 million polygons shows the fraction of unknots to be roughly $\nicefrac{9999}{10,000}$ for hexagons and $\nicefrac{2499}{2500}$ for heptagons\footnote{A total of $509$ knotted hexagons in $5$ million samples and $1791$ knotted heptagons in $5$ million samples were observed in our experiment.}. Our bounds even significantly underestimate the fraction of fixed-length hexagons and heptagons with total curvature less than $4 \pi$, which a similar experiment with 5 million samples reveals to be approximately $91.4\%$ and $63\%$, respectively. We could improve our bounds by computing the variance of total curvature, or even by finding an explicit expression for the total curvature pdf. However, this will involve a more subtle global analysis of the correlations between turning angles in closed polygons, and we leave this topic for future work.

%

%

\section{Acknowledgements}

The authors are happy to acknowledge the contributions of many friends and colleagues who held helpful discussions with us on polygons, knots, and probability, especially Tetsuo Deguchi and Ken Millett. We are also grateful to the Isaac Newton Institute and the Programme on Topological Dynamics in the Physical and Biological Sciences for hosting Cantarella, Kusner, and Shonkwiler, and to the Kavli Institute of Theoretical Physics and the Program on Knotted Fields for hosting Grosberg and Kusner while this paper was in preparation.


\bibliography{TotalCurv,TotalCurvExtra}

\begin{thebibliography}{10}

\bibitem{bikjalis}
Algimantas Bikjalis.
\newblock Asymptotic expansions for the densities and distributions of sums of
  independent identically distributed random vectors.
\newblock In {\em Selected Translations in Mathematical Statistics and
  Probability}, volume~13, pages 213--234. American Mathematical Society,
  Providence, RI, 1973.
\newblock Translation of Litovski\u\i\ Matematicheski\u\i\ Sbornik 8 (1968),
  405--422.

\bibitem{cpam}
Jason Cantarella, Tetsuo Deguchi, and Clayton Shonkwiler.
\newblock Probability theory of random polygons from the quaternionic
  viewpoint.
\newblock {\em Communications on Pure and Applied Mathematics}, 2013.
\newblock {T}o appear. \href{http://arxiv.org/abs/1206.3161}{\tt
  arXiv:1206.3161} {\tt [math.DG]}.

\bibitem{Diao:2001db}
Yuanan Diao, John~C Nardo, and Yanqing Sun.
\newblock Global knotting in equilateral random polygons.
\newblock {\em Journal of Knot Theory and its Ramifications}, 10(04):597--607,
  2001.

\bibitem{GrosbergExtra}
Alexander~Y Grosberg.
\newblock Total curvature and total torsion of a freely jointed circular
  polymer with {$n \gg 1$} segments.
\newblock {\em Macromolecules}, 41(12):4524--4527, 2008.

\bibitem{Knutson:2_iyExxE}
Jean-Claude Hausmann and Allen Knutson.
\newblock {Polygon spaces and Grassmannians}.
\newblock {\em L'Enseignement Math\'ematique. Revue Internationale. 2e
  S\'erie}, 43(1-2):173--198, 1997.

\bibitem{nist}
National {I}nstitute~of {S}tandards and {T}echnology.
\newblock Digital {L}ibrary of {M}athematical {F}unctions.
\newblock {\tt http://dlmf.nist.gov/}, March 23, 2012.

\bibitem{Gradshteyn:2007uy}
{Izrail S Gradshteyn} and Iosif~M Ryzhik.
\newblock {\em {Table of Integrals, Series, and Products}}.
\newblock Elsevier/Academic Press, Amsterdam, 7th edition, 2007.

\bibitem{Johnson:1994tg}
Norman~L Johnson, Samuel Kotz, and Narayanaswamy Balakrishnan.
\newblock {\em {Continuous Univariate Distributions. Vol. 1}}.
\newblock Wiley Series in Probability and Mathematical Statistics: Applied
  Probability and Statistics. John Wiley {\&} Sons Inc., New York, 2nd edition,
  1994.

\bibitem{JUNGREIS:1994cr}
Douglas Jungreis.
\newblock {Gaussian random polygons are globally knotted}.
\newblock {\em Journal of Knot Theory and its Ramifications}, 3(4):455--464,
  1994.

\bibitem{Lord:1954wh}
Reginald~Douglas Lord.
\newblock {The use of the Hankel transform in statistics I. General theory and
  examples}.
\newblock {\em Biometrika}, 41(1/2):44--55, 1954.

\bibitem{McLeish:1982fg}
Donald~Leslie McLeish.
\newblock {A robust alternative to the normal distribution}.
\newblock {\em The Canadian Journal of Statistics}, 10(2):89--102, 1982.

\bibitem{mrjcp}
Kenneth~C Millett and Eric~J Rawdon.
\newblock Energy, ropelength, and other physical aspects of equilateral knots.
\newblock {\em Journal of Computational Physics}, 186(2):426--456, 2003.

\bibitem{MR12:273c}
John~Willard Milnor.
\newblock {On the total curvature of knots}.
\newblock {\em Annals of Mathematics. Second Series}, 52:248--257, 1950.

\bibitem{Plunkett:2007vx}
Patrick Plunkett, Michael Piatek, Akos Dobay, John~C Kern, Kenneth~C Millett,
  Andrzej Stasiak, and Eric~J Rawdon.
\newblock {Total curvature and total torsion of knotted polymers}.
\newblock {\em Macromolecules}, 40(10):3860--3867, 2007.

\bibitem{Sumners:1999cd}
De~Witt Sumners and Stuart~G Whittington.
\newblock Knots in self-avoiding walks.
\newblock {\em Journal of Physics A: Mathematical and General},
  21(7):1689--1694, 1999.

\bibitem{Watson:1995ui}
George~Neville Watson.
\newblock {\em {A Treatise on the Theory of Bessel Functions}}.
\newblock Cambridge Mathematical Library. Cambridge University Press,
  Cambridge, 1995.

\end{thebibliography}

\end{document}